\documentclass{amsart}
\usepackage{pinlabel}
\usepackage{amsmath, amsthm, amssymb, amscd, mathrsfs, eucal, epsfig,color}
\usepackage{hyperref}
\usepackage{wrapfig}

\input xy
\xyoption{all}

\begin{document}

\newtheorem{theorem}{Theorem}[section]
\newtheorem*{homotopy_theorem}{Homotopy Theorem}
\newtheorem*{regular_value_theorem}{Regular Value Theorem}
\newtheorem{lemma}[theorem]{Lemma}
\newtheorem{corollary}[theorem]{Corollary}
\newtheorem{conjecture}[theorem]{Conjecture}
\newtheorem{proposition}[theorem]{Proposition}
\newtheorem{question}[theorem]{Question}
\newtheorem*{answer}{Answer}
\newtheorem{problem}[theorem]{Problem}
\newtheorem*{main_theorem}{Main Theorem}
\newtheorem*{claim}{Claim}
\newtheorem*{criterion}{Criterion}
\theoremstyle{definition}
\newtheorem{definition}[theorem]{Definition}
\newtheorem{construction}[theorem]{Construction}
\newtheorem{notation}[theorem]{Notation}
\newtheorem{convention}[theorem]{Convention}
\newtheorem*{warning}{Warning}
\newtheorem*{assumption}{Simplifying Assumptions}

\theoremstyle{remark}
\newtheorem{remark}[theorem]{Remark}
\newtheorem*{apology}{Apology}
\newtheorem{historical_remark}[theorem]{Historical Remark}
\newtheorem{example}[theorem]{Example}
\newtheorem{scholium}[theorem]{Scholium}
\newtheorem*{case}{Case}

\def\id{\text{id}}
\def\H{\mathbb H}
\def\Z{\mathbb Z}
\def\N{\mathbb N}
\def\R{\mathbb R}
\def\C{\mathbb C}
\def\CP{{\mathbb{CP}}}
\def\F{\mathbb F}
\def\FF{\mathcal{F}}
\def\P{\mathbb P}
\def\Q{\mathbb Q}
\def\M{{\mathcal M}}
\def\L{{\mathcal L}}
\def\Teich{{\mathcal T}}
\def\Mod{\text{Mod}}
\def\Cantor{\text{Cantor set}}
\def\erf{\text{erf}}
\def\CAT{\text{CAT}}
\def\SS{{\mathcal S}}
\def\BS{{\mathcal{BS}}}
\def\FS{{\mathcal{FS}}}
\def\GG{{\mathcal{G}}}
\def\FF{{\mathcal{F}}}
\def\EL{{\mathcal{EL}}}
\def\DL{{\mathcal{DL}}}
\def\Aut{\text{Aut}}
\def\Art{\text{Art}}
\def\Mon{\text{Mon}}
\def\CFS{\widehat{\mathcal{FS}}}
\def\CS{{\mathcal{CS}}}
\def\UCS{{\mathcal{UCS}}}
\def\E{{\mathbb E}}
\def\EE{{\mathcal E}}
\def\D{{\mathbb D}}
\def\roo{\text{root}}
\def\deg{\text{deg}}

\def\length{\textnormal{length}}

\newcommand{\marginal}[1]{\marginpar{\tiny #1}}

\title{Sausages and Butcher Paper}
\author{Danny Calegari}
\address{Department of Mathematics \\ University of Chicago \\
Chicago, Illinois, 60637}
\email{dannyc@math.uchicago.edu}
\date{\today}

\begin{abstract}
For each $d>1$ the {\em shift locus of degree $d$}, denoted $\SS_d$, 
is the space of normalized degree $d$ polynomials in one complex
variable for which every critical point is in the attracting basin of infinity
under iteration. It is a complex analytic manifold of complex dimension $d-1$.

We are able to give an explicit description of $\SS_d$ as a complex of spaces
over a contractible $\tilde{A}_{d-2}$ building, and to describe the pieces in
two quite different ways:
\begin{enumerate}
\item{(combinatorial): in terms of dynamical extended laminations; or}
\item{(algebraic): in terms of certain explicit `discriminant-like' 
affine algebraic varieties.}
\end{enumerate}

From this structure one may deduce numerous facts, including that $\SS_d$ has
the homotopy type of a CW complex of real dimension $d-1$; and that
$\SS_3$ and $\SS_4$ are $K(\pi,1)$s.

The method of proof is rather interesting in its own right. In fact, along the
way we discover a new class of complex surfaces (they are complements of 
certain singular curves in $\C^2$) which are homotopic to locally $\CAT(0)$ 
complexes; in particular they are $K(\pi,1)$s.
\end{abstract}

\maketitle

\setcounter{tocdepth}{1}
\tableofcontents

\section{Introduction}

For each $d>1$ the {\em shift locus of degree $d$}, denoted $\SS_d$, is the space
of normalized (i.e.\/ monic with roots summing to zero) 
degree $d$ polynomials in one complex variable for which every
critical point is in the attracting basin of infinity under iteration. A polynomial in
$\SS_d$ is called a {\em shift polynomial}. These are the polynomials whose
dynamics are the easiest to understand; perhaps in compensation, their parameter
spaces appear to be extremely complicated. Much is known about
the geometry and topology of $\SS_d$ and much is still mysterious.

The main point of this paper is to describe a canonical decomposition of $\SS_d$
(and some equivalent spaces) into pieces, giving $\SS_d$ the explicit structure of 
a `complex of spaces' over a rather nice space (a contractible $\tilde{A}_{d-2}$ 
building) and to give two, quite different, descriptions of the pieces.

One description is combinatorial, in terms of certain iterated fiber bundles
resp. their orbifolded quotients that we call {\em monkey prisms} resp. 
{\em monkey turnovers}. In this
description, the fibers and their monodromy are encoded quite explicitly in
objects called {\em dynamical elaminations}; the word `elamination' here is shorthand
for `extended lamination' --- a lamination with `extra' structure.
Elaminations are related to the
sorts of laminations used elsewhere in holomorphic dynamics (see e.g.\/ 
\cite{Thurston_dynamics}) but are in some ways quite different. Their definitions
and basic properties are given in \S~\ref{section:elaminations}, and they are
a key tool throughout the remainder of the paper.

The other description is algebraic, in terms of certain complex affine varieties,
which arise as moduli spaces of maps between infinite nodal genus 0 surfaces called
{\em sausages}. The relationship between sausages and shift polynomials is 
of an essentially topological nature, so that although both objects and their
moduli spaces carry natural complex analytic structures, the maps between
them do {\em not} respect this structure. This seems to be unavoidable: the shift
locus is a highly transcendental object, whereas the moduli spaces we construct
are algebraic. 

One interesting consequence of this relationship between these two ways of seeing 
the shift locus is that information which is obscure on one side can become 
transparent on the other. Here is an example. In degree 3, the Shift Locus can
be described (up to homotopy) as a space obtained from the 3-sphere by drilling
out a trefoil knot, and gluing in a bundle over $S^1$ whose fiber is a 
disk minus a Cantor set. This Cantor set can be thought of as an infinite nested
intersection $K= \cap E_n$ of subsets of the disk, where each $E_n$ 
is itself a finite union of disks. The monodromy permutes each $E_n$, and 
it is a fact (Theorem~\ref{theorem:powers_of_2}) that the orbits of this
permutation are cycles whose lengths are powers of 2. The only proof of this
that I know is to interpret the permutation action of the monodromy as 
an action on the roots of a certain polynomial obtained by
iterated quadratic extensions. 

\vskip 12pt

The value of mathematical machinery is that it can prove theorems whose 
statement does not mention the machinery. As a consequence of our structure 
theorems we are able to deduce some facts about the topology of the shift locus,
especially in low degrees. In particular:

\begin{homotopy_theorem}
$\SS_d$ has the homotopy type of a $(d-1)$-complex (i.e.\/ a complex of half the
real dimension of $\SS_d$ as a manifold). For $d=3$ or $4$ it is a $K(\pi,1)$.
For $d=3$ it is homotopic to a $\CAT(0)$ 2-complex.
\end{homotopy_theorem}

This is an amalgamation of Theorems~\ref{theorem:S_3_CAT0_2complex},
\ref{theorem:S_4_K_pi_1} and \ref{theorem:degree_d_homotopy_dimension}. In
fact, it is plausible that $\SS_d$ is a $K(\pi,1)$ in every degree.

In degree $3$ we are able to give an extremely explicit description.
$\SS_3$ is homeomorphic to a product $X_3 \times \R$ where $X_3$ is a
3-manifold obtained from $S^3$ by drilling out a right-handed trefoil, and
gluing in a bundle $D_\infty \to N_\infty \to S^1$ where each fiber $D_\infty(t)$ 
over $t \in S^1$ is a disk minus a Cantor set. In fact, we are able to give
a completely explicit description of $D_\infty$ and its monodromy in terms of
an object called the {\em Tautological Elamination}. There
is one Tautological Elamination $\Lambda_T(t)$ for each $t$. 
These elaminations vary continuously in 
the so-called collision topology (defined in \S~\ref{subsection:collision}),
and the $D_\infty(t)$ are obtained by an operation called {\em pinching}.
Finally, the monodromy is completely described by the formula 
$\FF_t \Lambda_T(s) = \Lambda_T(s+t)$ where $\FF_t$ is an explicit flow 
on the space of elaminations. 

In words: the monodromy on $D_\infty$ is the composition of infinitely
many fractional Dehn twists in a disjoint collection of circles, associated
to the elamination $\Lambda_T$ in a concrete manner. The combinatorics of
$\Lambda_T$ is rather complicated and beautiful; Theorem~\ref{theorem:powers_of_2}
and \S~\ref{subsection:sausage_tautological} describe some of its properties.

One intermediate result that we believe is interesting in its own right, 
is the discovery of a new class of affine complex surfaces which are $K(\pi,1)$s:
\begin{regular_value_theorem}
Let $Y_n$ be the space of degree 3 polynomials $z^3 + pz + q$ for which 
$n$ specific complex values (e.g.\/ the $n$th roots of unity) are regular values.
Then $Y_n$ is homotopic to a locally $\CAT(0)$ complex, and consequently is a $K(\pi,1)$.
\end{regular_value_theorem}

Even the case $n=2$ is new, so far as we know. 

\subsection{Apology}

`Butcher' in the title of this paper and throughout is a rather inelegant pun on the
name B\"ottcher which, Curt McMullen informs me, translates to {\em cooper} in 
English (i.e.\/ a maker of casks). However etymologically misguided, I 
have decided to keep `butcher' for the sake of the sausages.

\subsection{Other Work}

I would like to compare and connect the constructions and techniques in this paper
to prior and ongoing work of other mathematicians. First and foremost I would like 
to emphasize the resemblance of elements of the theory of dynamical elaminations
to the DeMarco--Pilgrim theory of {\em pictographs} as explained in 
\cite{deMarco_Pilgrim} (to the degree that I understand them). 
In fact, DeMarco, sometimes in collaboration with
Pilgrim or McMullen, has developed a sophisticated and intricate picture of the
shift locus over many years and papers; e.g.\/ \cite{deMarco_McMullen,deMarco}. 
The fact that $\SS_d$ has the homotopy type of a $(d-1)$-complex follows from
DeMarco's thesis \cite{deMarco_thesis}, where it is proved that $\SS_d$ is a Stein manifold.
I wish I better understood the relationship between her work and the point of
view we develop here. 

Recently, Blokh et. al. \cite{Blokh_et_al} have developed a theory of laminations 
to parameterize the pinching of components of (higher degree) Mandelbrot sets. 
I belive there is a family resemblance of their laminations to the 
tautological elamination we introduce in \S~\ref{subsection:tautological_lamination} 
and its variants and completions in higher degree, but the precise relationship is unclear.

The significance of configuration-space techniques (e.g.\/ braiding of roots,
attractors, etc.) to complex dynamics has been apparent at least since
the work of McMullen \cite{McMullen} and Goldberg--Keen \cite{Goldberg_Keen}.
This is a vast story that I only touch on briefly in 
\S~\ref{section:fundamental_groups}.

Branner--Hubbard \cite{Branner_Hubbard_2}, in a tour de force, found 
a detailed description of much of the parameter space of degree 3 polynomials.
In particular, they showed that $\SS_3$ (away from a piece with easily understood 
topology) has the structure of a bundle
over a circle (up to homotopy) whose fiber has free fundamental group. This is
perfectly parallel to our Theorem~\ref{theorem:S_3_topology}.
However, in their theory (which is more concretely tied to polynomials) the 
monodromy is completely opaque, and the culmination of their description 
(in \S~11.4) is only meant to indicate how formidable an explicit computation 
would be. Whereas in our theory, we have a completely explicit description of 
the fiber (it is the disk obtained by pinching the tautological elamination) 
and the monodromy (rotation by $\FF_t$).

\section{The shift locus}

Fix an integer $d>1$, and let $f(z)=\sum b_j z^{d-j}$ be a complex polynomial of degree
$d$, so that $b_0 \ne 0$. A change of variables $z \to \alpha z + \beta$ with $\alpha \in \C^*$ 
conjugates $f$ to a polynomial
$$f(z) = \sum \frac {b_j} {\alpha} (\alpha z + \beta)^{d-j} - \frac {\beta}{\alpha} \\
= \alpha^{d-1}b_0 z^d + \alpha^{d-2}(d\beta b_0 + b_1) z^{d-1} + \cdots
$$
Setting $\alpha = b_0^{1/(d-1)}$ and $\beta = -b_1/db_0$ we can put $f$ in {\em normal form}
$$f(z) = z^d + a_2 z^{d-2} + a_3 z^{d-3} + \cdots + a_d$$

There is non-uniqueness in the choice of $\alpha$; different choices differ by multiplication
by a $(d-1)$st root of unity $\zeta$, which multiplies the coefficient $a_j$ by 
$\zeta^{d-j-1}$.

\begin{definition}[Shift locus]\label{definition:shift_locus}
The {\em shift locus} of degree $d\ge 2$, denoted $\SS_d$, is the space of 
normalized degree $d$ polynomials $f$ for which every critical point of $f$
is in the attracting basin of infinity.
\end{definition}

The critical points of $f$ are the roots of $f'$. To say a point $c$ is in the attracting
basin of infinity means that the iterates $c, f(c), f^2(c), \cdots$ 
converge to infinity.

Note that the property of being in the shift locus is expressed in purely dynamical
terms. Thus we could define $\SS_d$ to be the space of {\em conjugacy classes}
of polynomials with a certain dynamical property. The relationship between that
definition and the one we adopt comes down to an ambiguity of $\Z/(d-1)\Z$
in the representation of a conjugacy class by a normalized polynomial. 

The coefficients of a normalized degree $d$ polynomial embed $\SS_d$ as a subset of $\C^{d-1}$.
It is clear that $\SS_d$ is open, since for any polynomial $f$
the punctured disk $E(R):=\lbrace z: |z| > R\rbrace$ is in the attracting basin of infinity
for sufficiently big $R$ (depending continuously on $f$), and $f$ is in $\SS_d$ if and only
if there is some integer $n$ so that $f^n(c) \in E(R)$ for all critical points $c$.

Recall the following definition:

\begin{definition}[Julia Set]
The {\em Julia set} $J_f$ of a polynomial $f$ is the closure of the set of repelling
periodic orbits of $f$.
\end{definition}

The complement of $J_f$ in the Riemann sphere is the {\em Fatou} set $\Omega_f$; it is
the maximal (necessarily open) set on which $f$ and all its iterates together form a
normal family. Actually, it is perhaps more natural to take this to be the
definition of the Fatou set, and to define the Julia set to be its complement.
The Julia set and the Fatou set are both totally invariant
(i.e.\/ $f(J_f)=J_f=f^{-1}(J_f)$ and similarly for $\Omega_f$). The Julia set is 
always nonempty and perfect. See e.g.\/ Milnor \cite{Milnor_CD}, \S~4.

\begin{proposition}
A polynomial $f$ is in the shift locus if and only if the Julia set $J_f$ 
is a Cantor set on which the action of $f$ is uniformly expanding.
\end{proposition}
\begin{proof}
If $J_f$ is a Cantor set, its complement is connected and is therefore equal to
the attracting basin of infinity. If $f$ is uniformly expanding on $J_f$ then 
$|f'|$ is bounded below on $J_f$ by a positive constant, so $J_f$ can't contain any
critical points and $f$ is in the shift locus.

Conversely, suppose $f$ is in the shift locus.
Since $\infty$ is an attracting fixed point, there is a connected 
neighborhood $U$ of $\infty$ with $f(U) \subset U$. Because $f$ is
a polynomial, $\infty$ is its own unique preimage under $f$; it follows by induction
that for each $n$, the set $V_n:=f^{-n}(U)$ is both forward-invariant and 
connected (because each component contains $\infty$). Because $f$ is in the shift locus, there
is an $n$ so that all the critical points are contained in $V_n$. Let $K$ be the complement
of $V_n$, so that $K$ is a finite union of disks.

Because all the critical points are in $V_n$, each point in $K$ has exactly $d$ distinct 
preimages; these vary continously as a function of $K$, and since each component $D$ of 
$K$ is simply-connected, $f^{-1}|D$ has $d$ well-defined continuous branches with disjoint 
image. By the Schwarz Lemma the branches of $f^{-1}$ are uniformly contracting
in the hyperbolic metric on each component of $K$; 
thus the diameters of the components of $f^{-n}(K)$ converge
(at a geometric rate) to zero, so that $\Lambda:=\cap_n f^{-n}(K)$ is totally disconnected 
and $f$ is uniformly expanding on $\Lambda$. 

Evidently the complement of $\Lambda$ is the basin of infinity, so $J_f=\Lambda$. Since
$J_f$ is always perfect, it is a Cantor set, 
and $f$ is uniformly expanding on $J_f$, as claimed.
\end{proof}

\begin{example}[Mandelbrot set]
A quadratic polynomial $z \to z^2 + c$ has $0$ as its unique critical point. The set
of $c \in \C$ for which $0$ is {\em not} in the basin of infinity of $z \to z^2+c$ is 
called the {\em Mandelbrot set} $\M$; see Figure~\ref{Mandelbrot}. Thus $\M$ is the
complement of $\SS_2$ in $\C$. The connectivity of the Mandelbrot set (proved by
Douady and Hubbard \cite{Douady_Hubbard_1}) is equivalent to the fact that $\SS_2$ is
homeomorphic to an (open) annulus.
\end{example}

\begin{figure}[htpb]
\centering
\includegraphics[scale=0.2]{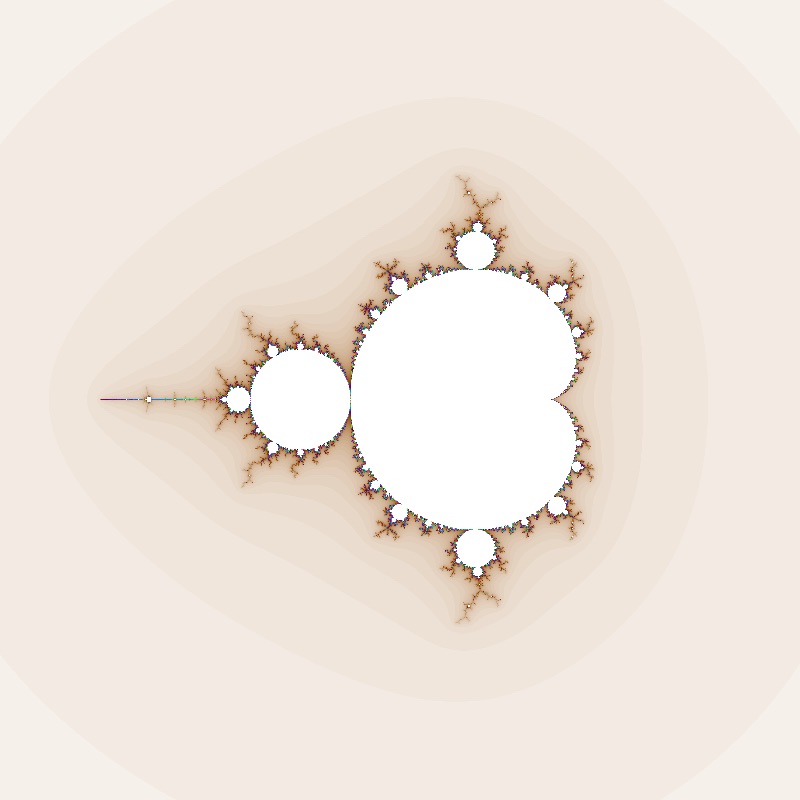}
\caption{The Mandelbrot set $\M$ (interior in white) is the complement of $\SS_2$ in $\C$}\label{Mandelbrot}
\end{figure}

\begin{example}[Discriminant complement]
Let $f(z)$ be any degree $d$ polynomial with distinct roots (i.e.\/ for which $0$ is
not a critical value). Then $g(z):=\lambda f(z)$ is (conjugate to a polynomial) 
in the shift locus for $|\lambda|\gg 1$.
To see this, let $U$ be any neighborhood of infinity for which $f(U)$ does not contain
$0$. Then for sufficiently large $|\lambda|$ we have $g(U) \subset U$ (so that $U$ is
contained in the attracting basin of infinity for $g$). Furthermore, $g$ and $f$
have the same critical points, so for sufficiently large $|\lambda|$ we have
$g(c) \in U$ for every critical point $c$ of $g$.

We can think of this as showing that near infinity, $\SS_d$ is `nearly equal' 
to the complement of the discriminant locus $\C^{d-1}-\Delta$. We shall elaborate on 
this remark in the sequel.
\end{example}

\begin{example}[Cantor set $J_f$]\label{example:Cantor_non_shift}
In degree two, $J_f$ is a Cantor set precisely when $f$ is in the shift locus, but
for degree bigger than two it is possible for $J_f$ to be a Cantor set for $f$ not
in the shift locus.

For example, consider the polynomial $f(z):=\alpha z(z-1)^2$ with $\alpha$ real
and positive. The fixed points are $0$ and $\beta^\pm:=1 \pm \sqrt{1 -(\alpha-1)/\alpha}$ and
the critical points are $1/3$ and $1$. Since $f(1)=0$, the polynomial $f$
is never in the shift locus. If $f(1/3)>\beta^+$ then 
$f^{-1}([0,\beta^+])$ is real and properly contained in $[0,\beta^+]$, and
$J_f= \cap_n f^{-n}([0,\beta^+])$ is a totally real Cantor set. This happens
for $\alpha>9$.

In the limiting case $\alpha = 9$, the Julia set $J_f$ is the real interval $[0,4/3]$.
\end{example}

Suppose $f$ is in the shift locus, so that $J_f$ is a Cantor set, equal to the
complement of the basin of infinity. Then $f$ has $d$ distinct fixed points, all in
$J_f$.

Because the dynamics of $f$ on $J_f$ is expanding, it is structurally stable there.
So if $f_t$ is a family of polynomials in the shift locus with Julia sets
$J_{f_t}$, there are open sets $U(t)$ containing $J_{f_t}$ and maps
$\varphi_t:U(0) \to U(t)$ conjugating $f_t|U(t)$ to $f_0|U(0)$. In particular, we
obtain a {\em monodromy representation} $\rho$ from the fundamental group 
$\pi_1(\SS_d)$ to the {\em mapping class group} of $\C - \Cantor$. This is an 
example of a so-called {\em big mapping class group}; see e.g.\/ \cite{Vlamis} 
for background and an introduction to the theory of such groups.

The dynamics of any $f$ on $J_f$ is conjugate to the action of
the shift on the space of one-sided sequences in a $d$ letter alphabet; this
justifies the name. One way to see this
is to take a compact $K$ containing $J_f$ in its interior for which $f|K$ has $d$ inverse
branches $f_1,\cdots,f_d$, and the $f_j(K)$ are disjoint subsets of the interior of $K$.
Then $J_f$ is in bijection with the set of right infinite words in the $\lbrace f_j\rbrace$.

\medskip

The geometry of $\SS_d$ is very complicated. For $d=2$ the space 
$\SS_2$ is the complement in $\C$ of the Mandelbrot set; showing 
that $\SS_2$ is conformal to a punctured disk is equivalent to showing 
that the Mandelbrot set is connected. The main goal of this paper is to
develop tools to describe the topology of $\SS_d$ for higher $d$.

\section{Elaminations}\label{section:elaminations}

In this section we introduce the concept of an {\em elamination}. Laminations, as
introduced by Thurston, are a key tool in low-dimensional geometry, topology
and dynamics; see e.g.\/ \cite{Thurston_notes}, Chapter~8.5. The reader already
familiar with laminations can think of the term `elamination' as an abbreviation
for `extended lamination', or `enhanced lamination' --- an ordinary lamination
with some extra structure. 

Elaminations are an essential combinatorial tool that will be used throughout 
the sequel, especially beginning with \S~\ref{section:butcher_paper}, so throughout 
this section we just spell out the basic theory, deferring the connection to 
dynamics until the sequel. There are some points of contact between elaminations
--- and in particular the `collision topology' on the space $\EL$ ---
to the theory partially developed by Thurston in \cite{Thurston_laminations}; but 
there are many points of difference, and it seems pointless to try to force
the two theories into a common framework.

Elaminations (and laminations for that matter) have several 
more-or-less equivalent identities, and it is useful to be
able to move back and forth between them. By abuse of notation, we will often use the
same symbol or term to refer to the underlying abstract object or any of its equivalent
manifestations. 

We fix the following notation here and throughout the rest of the paper: let $\D$ 
denote the {\em closed} unit disk in the complex plane $\C$, and let $\E:=\C-\D$
denote its {\em open} exterior.

\begin{definition}[Circle Lamination]
A {\em leaf} is a finite subset of the unit circle of cardinality at least $2$. A
leaf is {\em simple} if it consists of 2 points; a leaf of {\em multiplicity $n$} 
consists of $n+1$ points.

A {\em circle lamination} is a set of leaves, no two of which have 2 element
subsets that are linked. A circle lamination is simple if all its leaves are simple.
\end{definition}

Most authors require laminations to be closed in the space of finite subsets of
$S^1$ (in the Hausdorff topology), but we explicitly do {\em not} require this.

\begin{definition}[Geodesic Lamination]
A {\em simple geodesic leaf} is a complete geodesic in $\D$ with its hyperbolic metric.
A {\em geodesic leaf of multiplicity $n>1$} is an ideal $(n+1)$-gon.

A {\em geodesic lamination} is a set of geodesic leaves no two of which cross in
$\D$. A geodesic lamination is simple if all its leaves are simple.
\end{definition}

Every ideal $(n+1)$-gon in $\D$ determines an unordered set of $n+1$ endpoints in $S^1$ and
conversely. Two $(n+1)$-gons in $\D$ cross if and only if two pairs of their endpoints 
link in $S^1$. Thus there is a natural correspondence between circle laminations 
and geodesic laminations.

\begin{definition}[Elamination]\label{definition:elamination}
For each $z \in \E$ we let $\ell(z)$ denote the straight line segment from 
$z/|z|$ to $z$. We call $\ell(z)$ a {\em radial segment}. 
The {\em height} of the segment $\ell(z)$ is $\log(|z|)$.

An {\em extended leaf} of {\em height $h>0$}
is the union of a geodesic leaf in $\D$ (the {\em vein}) with
radial segments in $\E$ (the {\em tips}) all of height $h$,
attached at the endpoints of the vein. An extended leaf is {\em simple} 
if the vein is simple.

An {\em extended lamination}, or {\em elamination} for short, is a set of extended
leaves with the following properties:
\begin{enumerate}
\item{lamination: distinct leaves have distinct veins, and 
the set of all veins of all leaves forms a geodesic lamination (called the {\em vein}
of the elamination);}
\item{properness: there are only finitely many extended leaves with height $\ge \epsilon$ 
for any $\epsilon>0$ (thus every elamination has only countably many leaves); and}
\item{saturation: to be defined below.}
\end{enumerate}
\end{definition}

Let us now explain the meaning of saturation. Let $\Lambda$ be an elamination, and let
$\ell$ be a leaf with height $h$. Let $pq$ be an oriented edge of $\ell$, and
let $L$ be the finite set of leaves of $\Lambda$ on the positive side of $pq$ with
height $\ge h$. Let $L_p$ (resp. $L_q$) denote the subset of $L$ of leaves with
an endpoint with the same argument as $p$ (resp. $q$). Since leaves of
$\Lambda$ do not cross, and distinct leaves have distinct veins, the leaves $L_p$
are ordered by how they separate each other from $pq$; thus if $L_p$ is nonempty
there is a {\em closest} $\ell_p \in L_p$ to $pq$ (and similarly for $L_q$).
A leaf $\ell_p$ (resp. $\ell_q$) if it exists, is called an {\em elder sibling} for
$\ell$ at $p$ (resp. at $q$).

Saturation means the following two conditions hold for every $\ell$:
\begin{enumerate}
\item{an elder sibling of $\ell$ has height $h'$ strictly bigger than $h$; and}
\item{if $L_p$ is nonempty so is $L_q$ and vice versa; and
furthermore $\ell_p=\ell_q$.}
\end{enumerate}
We say that a leaf $\ell$ is {\em saturated} by an elder sibling.
Another way to say this is that if the vein of $\ell$ shares one endpoint with
the vein of a taller leaf $\ell'$, and there are no other $\ell''$ (also taller than $\ell$)
in the way, then the vein of $\ell$ actually shares two endpoints with $\ell'$.

\subsection{Pinching}

Let $\Lambda$ be an elamination. We define an operation called {\em pinching}
that associates to $\Lambda$ a Riemann surface $\Omega$ obtained from $\E$ by suitable
cut and paste along the tips of $\Lambda$.

\begin{construction}[Pinching]
Let $\Lambda$ be an elamination. For each leaf $\lambda$ with multiplicity $n$ and
with tips $\sigma_0,\cdots,\sigma_n$ enumerated in cyclic order in $S^1$,
cut open $\E$ along the $\sigma_j$ and glue the right side of
each $\sigma_j$ to the left side of $\sigma_{j-1}$ (indices taken mod $n+1$) 
by a Euclidean isometry.

The resulting Riemann surface $\Omega$ is said to be obtained from $\Lambda$ by 
{\em pinching}. We also write $\Omega = \E \mod \Lambda$.
\end{construction}

\begin{lemma}[Planar]
$\Omega$ obtained from an elamination $\Lambda$ by pinching is planar.
\end{lemma}
\begin{proof}
This is equivalent to the fact that the leaves do not cross.
\end{proof}

By construction, the function $\log{|\cdot|}:\E \to (0,\infty)$ is
preserved under pinching, and therefore descends to a well-defined proper function
on $\Omega$ that we refer to as the {\em height function} or sometimes as
the {\em Green's function}, and denote $h$. 
Furthermore, $d\arg$ is a well-defined 1-form on 
$\Omega$, so the level sets of the height function are finite unions of metric graphs.
We sometimes denote $d\arg$ by $d\theta$. In fact, the combination $dh+id\theta$ is
just the image of $d\log(z)$ on $\E$, which makes sense because this 1-form is
preserved by cut-and-paste. By abuse of notation therefore we sometimes write
$dh+id\theta = d\log(z)$. This $1$-form has 
a zero of multiplicity $m$ for each leaf of multiplicity $m$.

\begin{definition}[Monkey pants]\label{definition:monkey_pants}
A {\em monkey pants} is a (closed) disk with at least two (open) subdisks removed.
If $P$ is a monkey pants, a function $\pi:P \to [t_1,t_2]$ is {\em monkey Morse} if it is a
submersion away from finitely many points in the interior which are all saddles or
monkey saddles, and if $\pi^{-1}(t_2)$ is equal to a distinguished boundary component
$\partial^+ P$ (the {\em waist}) and $\pi^{-1}(t_1)$ is equal to the other 
components $\partial^- P$ (the {\em cuffs}).
\end{definition}

Let $\Omega$ be the Riemann surface associated to an elamination.
If $0 < t_1 < t_2$ are numbers not equal to the height of any leaf, then
$\Omega([t_1,t_2]):= h^{-1}[t_1,t_2] \subset \Omega$ is a monkey pants, 
and $h$ restricted to $\Omega([t_1,t_2])$ is monkey Morse. There is one saddle
point for each simple leaf with height in $[t_1,t_2]$, and one monkey saddle with
multiplicity equal to the multiplicity of a non-simple leaf.

Suppose $\Lambda$ is a finite elamination, which pinches $\E$ to $\Omega$. Then
$\Omega$ is a plane minus $n+1$ disks, where $n$ is the number of leaves of 
$\Lambda$ counted with multiplicity. 
If $t$ is the least height of leaves of $\Lambda$, then $\Omega((0,t))$ is
a disjoint union of $n+1$ annuli whose inner `boundary components' 
(where $h \to 0$) can be compactified by $n+1$ circles. We refer to this collection
of circles as $S^1 \mod \Lambda$. Thus: just as $\E$ is compactified (away from $\infty$)
by $S^1$, the surface $\E \mod \Lambda$ is compactified (away from $\infty$) by 
$S^1 \mod \Lambda$.

\subsection{Push over and amalgamation}\label{subsection:collision}

Denote the set of elaminations by $\EL$. We would like to define a natural topology on
$\EL$. In a nutshell, a family of elaminations $\Lambda_t$ in $\EL$ 
varies continuously if and only if the Riemann surfaces $\Omega_t=\E \mod \Lambda_t$ 
do. 

Because of properness, an elamination $\Lambda$ has only finitely many leaves of height
bigger than any positive $\epsilon$. When these leaves have disjoint veins, it is 
obvious what it means to say that they vary continuously in a family: it just means
that the heights and arguments vary continuously.

When two leaves of different heights collide, the shorter leaf becomes 
{\em saturated} by the taller (which becomes at that moment its elder sibling); 
if we continue the motion in the obvious way, the shorter leaf
becomes unsaturated as it moves away from the taller leaf, and the net result is that
the shorter leaf has been {\em pushed over} the taller one.
The meaning of this is illustrated in Figure~\ref{saturated_and_push_over}.

\begin{figure}[htpb]
\labellist
\small\hair 2pt
\pinlabel $\rightsquigarrow$ at 0 0
\endlabellist
\centering
\includegraphics[scale=0.6]{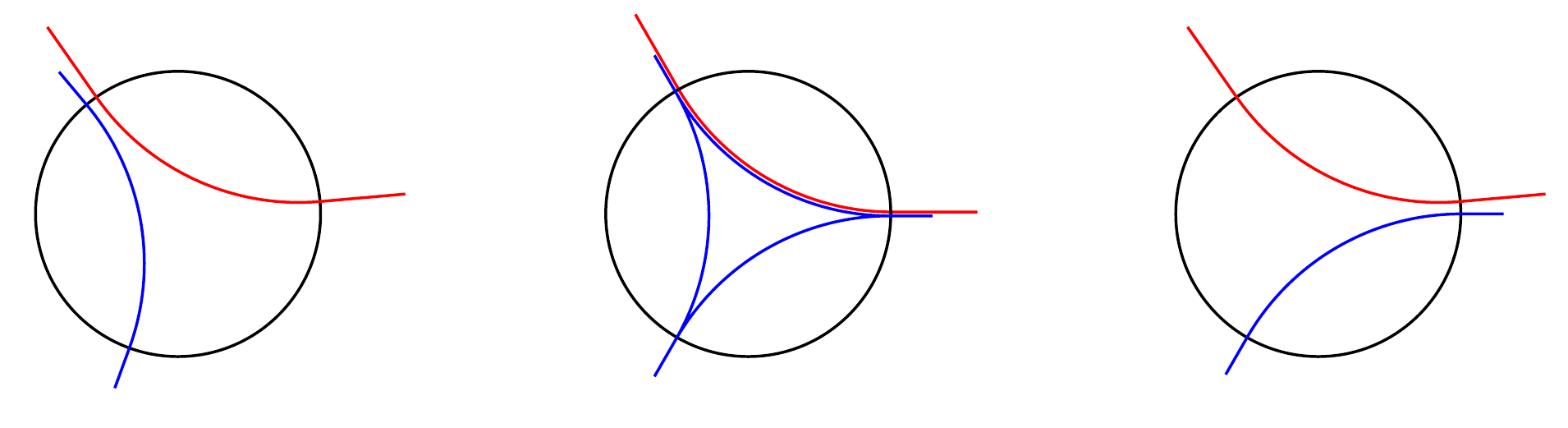} 
\caption{Pushing a shorter leaf over a taller one; at the intermediate step the shorter leaf
is saturated by the taller one}\label{saturated_and_push_over}
\end{figure}

When two leaves of the same height collide, saturation dictates that they must become
amalgamated into a common leaf; see Figure~\ref{amalgamate}.

\begin{figure}[htpb]
\labellist
\small\hair 2pt
\pinlabel $\rightsquigarrow$ at 0 0
\endlabellist
\centering
\includegraphics[scale=0.6]{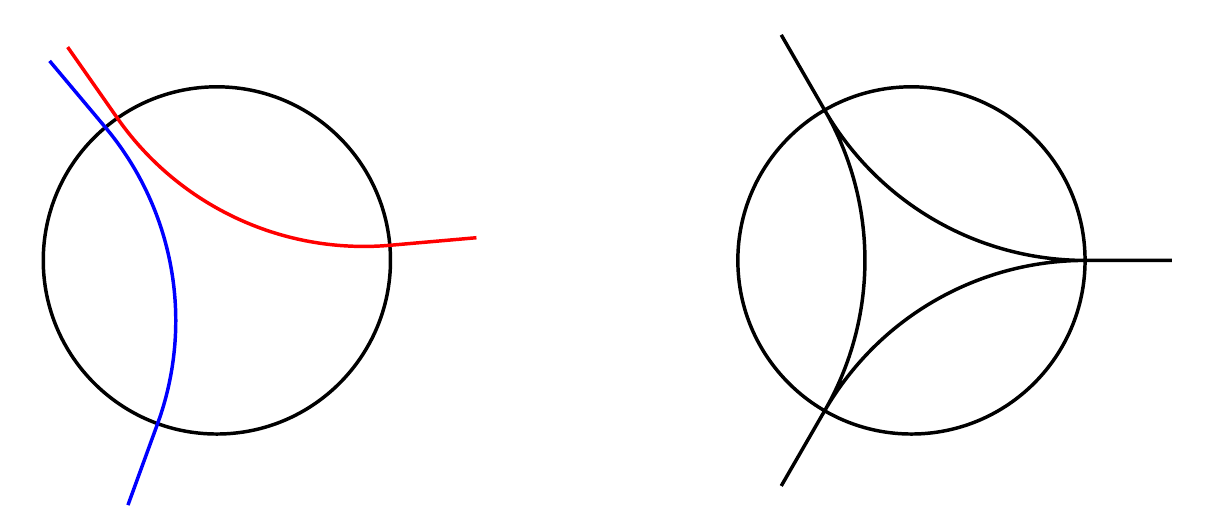} 
\caption{When two simple leaves of the same height collide, they amalgamate 
to form a leaf of multiplicity 2}\label{amalgamate}
\end{figure}

We now define a topology on $\EL$ called the {\em collision topology}.
\begin{definition}[Collision Topology]
A family of elaminations $\Lambda_t$ varies {\em continuously} in $\EL$ in the
collision topology if every finite subset of leaves varies continuously 
when they are disjoint, and varies by push over or amalgamation when they collide.
\end{definition}

The whole point of the collision topology is that it is compatible with pinching.

\begin{lemma}[Continuous quotient]
If $\Lambda_t$ varies continuously in $\EL$ then $\Omega_t$ vary continuously as
Riemann surfaces.
\end{lemma}
\begin{proof}
The only thing to check is that push over and amalgamation are continuous under
pinching; but this is essentially by definition.
\end{proof}

\section{Butcher Paper}\label{section:butcher_paper}

\subsection{B\"ottcher Coordinates}

Let $f(z):=z^d+a_2z^{d-2}+ \cdots + a_d$ be a degree $d$ polynomial in normal form.
Lucjan B\"ottcher, a Polish mathematician who worked in Lvov in the beginning
of the 20th century, showed \cite{Bottcher}
that $f$ is conjugate to $z\to z^d$ in a neighborhood of infinity:

\begin{proposition}[B\"ottcher Coordinates]\label{bottcher_coordinates}
Let $f(z):=z^d+a_2z^{d-2}+ \cdots + a_d$ be a degree $d$ polynomial in normal form.
Then $f$ is holomorphically conjugate to $z \to z^d$ on some neighborhood of infinity.
\end{proposition}

For a proof see e.g.\/ Milnor \cite{Milnor_CD}, Thm.~9.1.

\subsection{Holomorphic 1-form}

Let's let $\phi$ be the holomorphic conjugacy promised by 
Proposition~\ref{bottcher_coordinates} normalized so that $\phi f \phi^{-1}(z) = z^d$
near infinity. The map $\phi$ is only defined in a neighborhood of infinity, but we
can extend it inductively over larger and larger domains by using the functional
equation. Recall that $\E$ denotes the exterior of the closed unit disk in $\C$;
i.e.\/ $\E$ is the basin of infinity of $z \to z^d$. The function $\log{z}$ is not
single-valued on $\E$, but its differential $dz/z$ is. The map $z \to z^d$ 
pulls back $dz/z$ to $d\cdot dz/z$ (we use the notation $d\cdot$ to indicate
multiplication by the degree $d$ to distinguish it from the exterior derivative of forms).
If we define $\alpha:=\phi^* dz/z$ in a neighborhood of infinity, 
we can extend $\alpha$ uniquely to all of the Fatou set $\Omega_f$ by iteratively solving 
$f^*\alpha = d\cdot \alpha$. Thus $\alpha$ is a holomorphic $1$-form on $\Omega_f$ 
with zeroes at the critical points of $f$ and their preimages.

\subsection{Horizontal/Vertical foliations}\label{horizontal_vertical_subsection}

The real and imaginary parts of $\alpha$ and $dz/z$ give rise to foliations on
$\Omega_f$ and on $\E$ related by $\phi$ near infinity. We call these the {\em horizontal} 
and the {\em vertical} foliations respectively. 

On $\E$ these foliations are nonsingular;
the horizontal leaves are the circles $|z|=\text{constant}$ and the vertical leaves are
the rays $\text{arg}(z)=\text{constant}$. The corresponding foliations
on $\Omega_f$ have saddle singularities at simple critical points and their preimages,
and monkey saddle singularities at critical points (and their preimages) of multiplicity
bigger than one (as roots of $f'$). Evidently $\phi$ may be extended by analytic
continuation along every nonsingular vertical leaf, and along every singular leaf from
infinity until the first singularity. These singularities are critical points and
their preimages; this is a proper subset of $\Omega_f$.

\subsection{Construction of the dynamical elamination}\label{subsection:dynamical_elamination}

Let $L_f \subset \Omega_f$ be the complement of this (maximal) domain of definition of $\phi$, 
and $L \subset \E$ the complement of $\phi(\Omega_f - L_f)$. These subsets are both
closed and backwards invariant. The complements $\Omega_f - L_f$ and $\E - L$ are 
open, simply connected, and dense. The set $L$ consists of a countable collection of
radial segments; in the generic case there are exactly two such segments
$\ell(q^\pm)$ for each critical or pre-critical point $p$. One may think of 
$q^\pm$ as the `image' of $p$ under $\phi$. If $c$ is a simple critical point with image
$v=f(c)$ then $\phi(v)$ will have $d$ preimages under $z \to z^d$, whereas $v$ will only
have $d-1$ preimages under $f$; the two of the preimages of $\phi(v)$ that 
correspond to $c$ are $q^\pm$.

\begin{example}
If $f$ has real coefficients, $\phi$ preserves the real axis. Thus the vertical
leaves with $\arg(\phi(z))\in \pi d^{-n} \Z$ consist of the $z$ with $f^n(z)$ real. 
The polynomial $f(z):=z^3 + 3z + 3^{-1/2}$ has critical points at $\pm i$ with
initial forward orbit 
$$\pm i \to 3^{-1/2}\pm 2i \to -23\cdot 3^{-3/2} \approx -4.42635$$

Figure~\ref{landing_rays} shows some vertical leaves in $\Omega_f$ and in $\E$ in
the preimage of the negative real axis. $L_f$ and $L$ are in red. The set
$L_f \cup J_f$ is a dendrite.
\begin{figure}[htpb]
\centering
\includegraphics[scale=0.4]{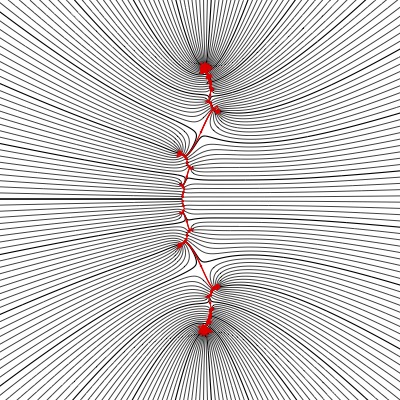} \hskip 24pt
\includegraphics[scale=0.4]{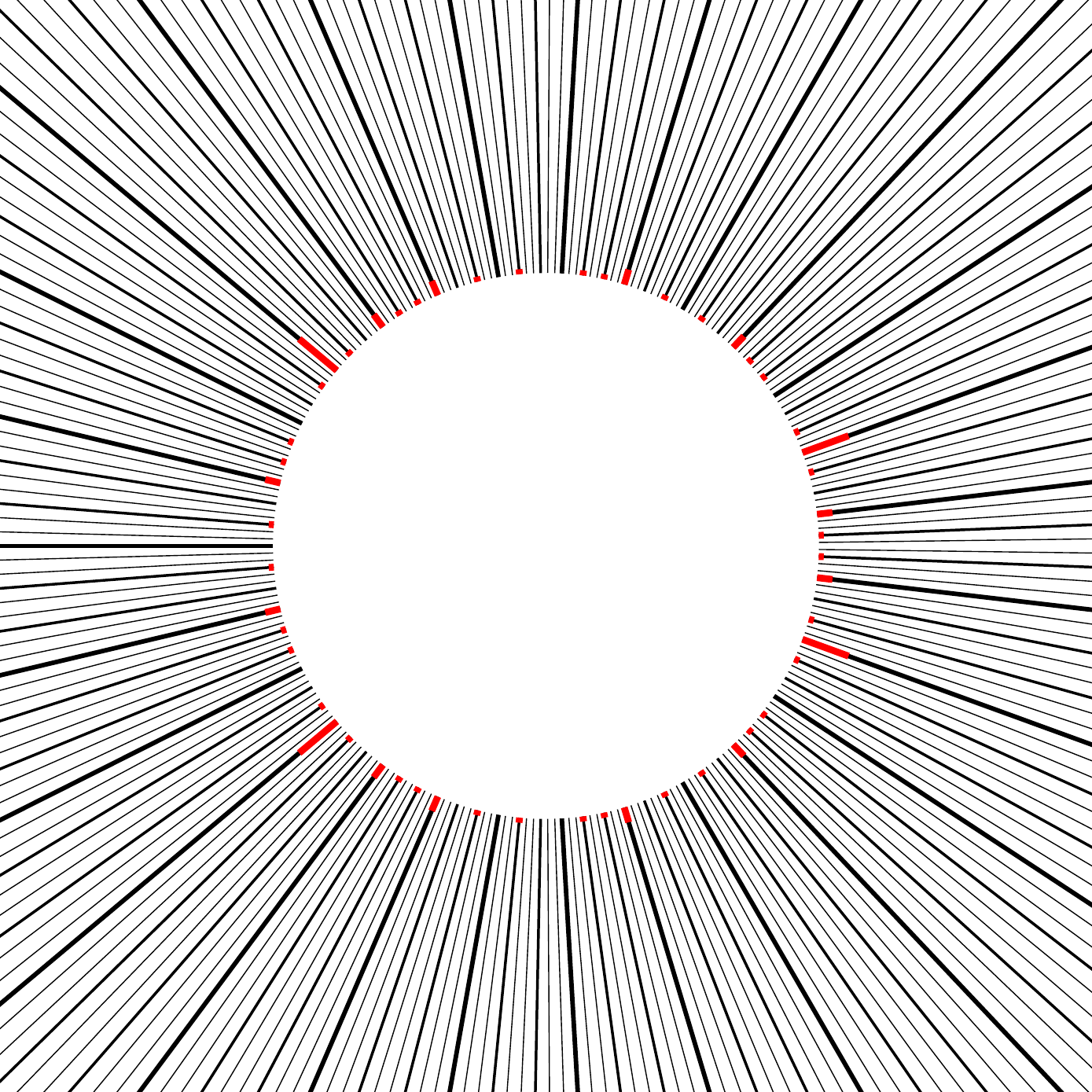}
\caption{Vertical leaves in $\Omega_f$ and in $\E$ for $f(z):=z^3 + 3z + 3^{-1/2}$}\label{landing_rays}
\end{figure}

Note that $\arg(\phi(f^2(i)))=\pi$ and 
$\arg(\phi(f(i)))=\pi/3$. The absolute value $|\phi(i)|$ is well-defined,
and equal to approximately $1.18$, but $\arg(\phi(i))$ is multi-valued, and takes
values $7\pi/9$ and $\pi/9$.
\end{example}

One may repair this multi-valuedness of $\phi$ by doing cut-and-paste on $\E$: cut
open $\E$ along the segments $L$ and reglue edges in pairs, so that each copy of
$\ell(q^+)$ is glued to a copy of $\ell(q^-)$ in the unique manner which is
orientation-reversing and compatible with the dynamics $z \to z^d$. The result is a
new Riemann surface $\Omega$ on which the map $z \to z^d$ on $\E-L$ extends uniquely 
to a holomorphic degree $d$ map $F: \Omega \to \Omega$ 
and for which $\phi:\Omega_f -L_f \to \E-L$ extends to a holomorphic
isomorphism $\phi:\Omega_f \to \Omega$ conjugating $f$ to $F$. 

Another way to say this is that $L$ is the set of tips of a simple elamination $\Lambda$,
with one leaf for each pair $\ell(q^\pm)$. And $\Omega$ is precisely the Riemann
surface obtained from $\Lambda$ by pinching, together with the 1-form $dz/z$ whose
real and imaginary parts are the (derivatives of) height and argument respectively.

When one talks about constructing a Riemann surface by gluing Euclidean polygons,
one sometimes says the Riemann surface is built `from paper' 
(see e.g.\/ \cite{Carvalho_Hall}).
As a mnemonic therefore, and by abuse of homonymy, we say that $\Omega$ is built from 
{\em butcher paper}.

In case some critical points are not simple, there might be three (or more) segments
in $L$ associated to some (pre)-critical points, and some segment $\ell(q^\pm)$ associated
to a critical point $c$ might be a subsegment of some precritical $\ell(r^\pm)$ associated
to another critical point. Exactly as in the simple case, these sets form the tips of
the leaves of an elamination $\Lambda$ (no longer simple) and $\Omega=\E \mod \Lambda$.

\begin{definition}[Dynamical Elamination]
The elamination $\Lambda$ obtained from $f$ as above is called the {\em dynamical
elamination} associated to $f$.
\end{definition}

If we need to stress the dependence of $\Lambda$ on $f$ we denote it $\Lambda(f)$.

\begin{lemma}
The assignment $\Phi:f \to \Lambda(f)$ is a continuous function from $\SS_d$ to $\EL$
that we call the {\em butcher map}.
\end{lemma}
\begin{proof}
The Fatou sets $\Omega_f$ together with their vertical/horizontal foliations vary
continuously as a function of $f$. Since $\Lambda(f)$ can be recovered from $\Omega_f$
under the identification of $\E \mod \Lambda(f)$ with $\Omega_f$, and since we defined
the topology on $\EL$ so that the inverse of pinching is continuous, the lemma follows.
\end{proof}

\section{Formal shift space}\label{formal_shift_space}

In this section we shall characterize the dynamical elaminations $\Lambda(f)$
that arise from shift polynomials by the construction in 
\S~\ref{subsection:dynamical_elamination}, and describe an inverse map. The
existence of this inverse is the Realization Theorem~\ref{theorem:realization},
due essentially to DeMarco--McMullen, although we express things in rather 
different language.

In this section we use logarithmic coordinates and
fix the notation $\log(z)=r+i\theta$ for
$z\in \E$, so that $r \in \R^+$ and $\theta \in \R/2\pi\Z$, and we denote
the radial segment associated to $z$ by $\ell(r,\theta)$. In $(r,\theta)$ coordinates,
the map $z \to z^d$ acts as multiplication by $d$. We call $r$ the {\em height}
and $\theta$ the {\em angle} of the segment $\ell(r,\theta)$. 

\subsection{Dynamical Elaminations}

The geometry and combinatorics of $L$ is best expressed in the language of
elaminations. Let's fix the degree $d$ in what follows.

\begin{definition}[Critical data]
A (degree $d$){\em critical leaf} is an extended leaf whose tips have angles that are
equal mod $2\pi d^{-1}$. 

If $C_1,\cdots,C_e$ is a finite set of degree $d$ critical leaves, we say the
{\em critical multiplicity} of $C_j$ is equal to its ordinary multiplicity, minus
$1$ for every $C_k$ with greater height which shares a pair of ideal points with $C_j$.

A (degree $d$) {\em critical set} is a finite elamination consisting of 
degree $d$ critical leaves $C_1,\cdots,C_e$ whose critical multiplicities sum to $d-1$.
\end{definition}

The map $z \to z^d$ acts on radial segments by $\ell(r,\theta) \to \ell(dr,d\theta)$.
This induces a (partially) defined action on extended leaves, that might reduce multiplicity
if distinct tips have angles that differ by a multiple of $2\pi d^{-1}$. If $\lambda$
is a leaf for which all tips have angles that differ by a multiple of $2\pi d^{-1}$,
the image of $\lambda$ under $z \to z^d$ is undefined. For instance, $z \to z^d$ is
undefined on any critical leaf. If $P$ is a leaf, we denote its image under $z \to z^d$
by $P^d$.

\begin{definition}[Dynamical Elamination]
A {\em dynamical elamination} $L$ is an elamination containing a finite subset of leaves
$C$ which is a degree $d$ critical set, and such that $z \to z^d$ maps 
$L-C$ to $L$ in a $d$ to $1$ manner. We say $L$ is {\em generated by} $C$.
\end{definition}

Figure~\ref{dynamical_lamination} indicates a simple dynamical elamination of degree 3. 
 
\begin{figure}[htpb]
\centering
\includegraphics[scale=0.8]{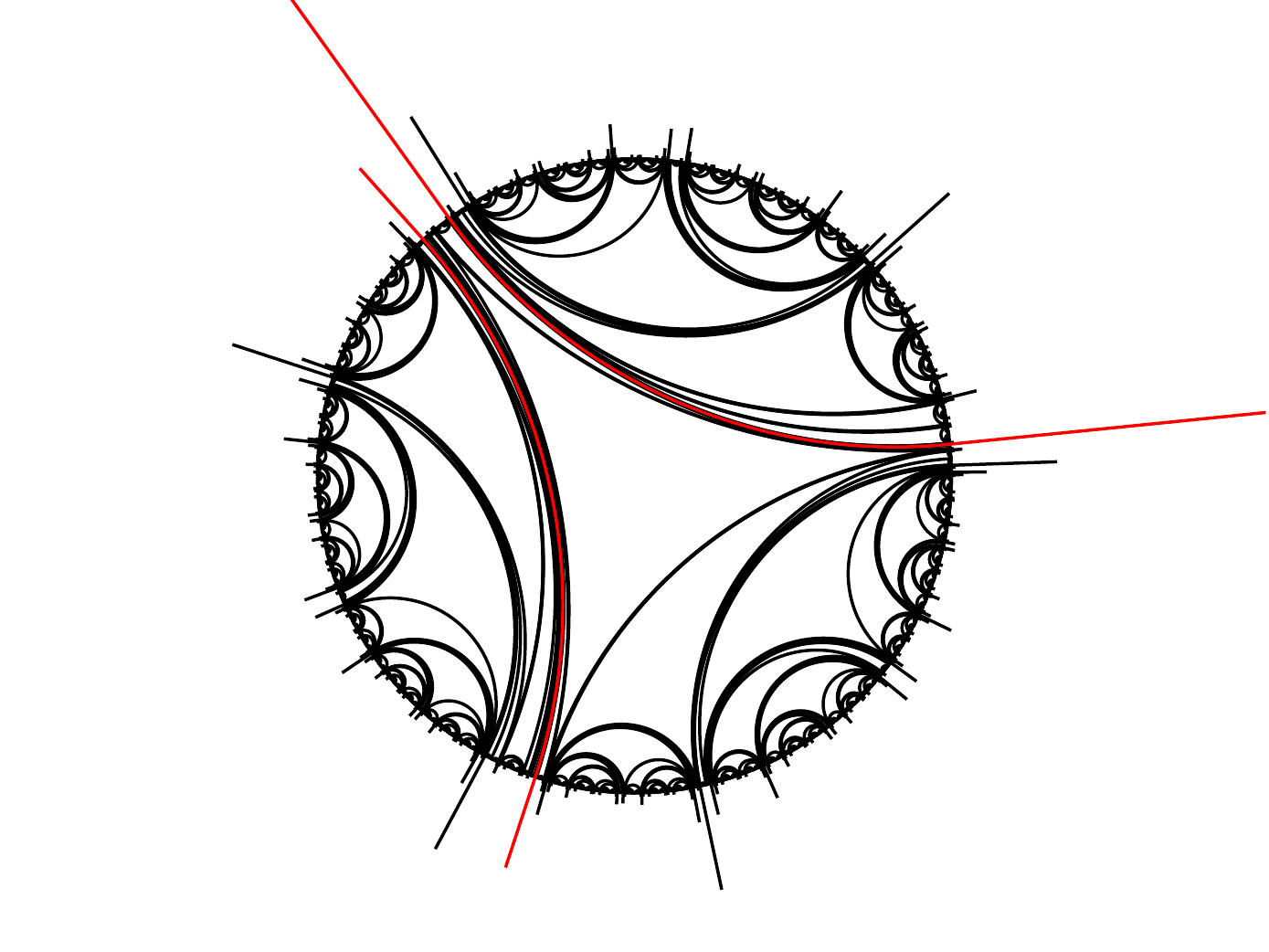} 
\caption{Simple dynamical elamination of degree $3$; critical leaves are in red}\label{dynamical_lamination}
\end{figure}

\begin{proposition}[Dynamical elamination]\label{proposition:dynamical_lamination}
Let $C$ be a degree $d$ critical set. Then there is a unique dynamical elamination $L$
generated by $C$.
\end{proposition}
\begin{proof}
Recall that the notation $S^1 \mod C$ denotes the result of 
pinching the unit circle along $C$. From the definition of a critical set,
$S^1 \mod C$ is the union of $d$ disjoint circles, each canonically
isomorphic to $\R/\frac{1}{d}\Z$ (with respect to the angle coordinates it inherits
from $S^1$). Thus the map $z \to z^d$ maps each of these circles
isomorphically to the unit circle. An extended leaf in $S^1 \mod C$
canonically pulls back to an extended leaf on the unit circle by taking the preimage
of the tips to be the tips of the preimage. We may therefore inductively construct $L$ as
the union of $L_n$ where $L_0=C$ and $L_j$ is obtained from $L_{j-1}$ by taking the
preimages of $L_j$ in $S^1 \mod C$ and pulling back to an elamination on $S^1$.
Uniqueness is clear.
\end{proof}

We refer to the preimages of the critical leaves as {\em precritical leaves},
and we say that the {\em depth} of a precritical leaf $P$ is the number of iterates
of the dynamical map which take it to some $C_i$.

\subsection{Realization}

Let $L$ be a degree $d$ dynamical elamination generated by $C$, and let $\Omega$ be the
Riemann surface obtained from $L$ by pinching. The map $z \to z^d$ induces a degree
$d$ proper holomorphic map $F$ from $\Omega$ to itself with $d-1$ critical points
counted with multiplicity, which are the endpoints of the tips of the $C$.

The {\em Realization Theorem} says that the action of $F$ on $\Omega$ is 
holomorphically conjugate to the action of some (unique) shift polynomial $f$ 
on its Fatou set.

\begin{theorem}[Realization]\label{theorem:realization}
Let $C$ be a degree $d$ critical set with dynamical elamination $L$ and 
associated Riemann surface $F:\Omega \to \Omega$. Then there is a unique conjugacy
class of degree $d$ polynomial $f$ in the shift locus for which $f|\Omega_f$ is
holomorphically conjugate to $F|\Omega$. 
\end{theorem}

Essentially the same theorem is proved by DeMarco--McMullen \cite{deMarco_McMullen},
Thm.~7.1 although in different language, and with quite a different proof. 
One distinctive feature of our proof 
of Theorem~\ref{theorem:realization} is that it finds the desired embedding of 
$\Omega$ in $\CP^1$ by a rapidly convergent algorithm; we expect this
might be useful e.g.\/ for computer implementation. 

\begin{proof}
The Riemann surface $\Omega$ has one isolated puncture (corresponding to $\infty$)
and a Cantor set $J$ of ends (the `image' of the unit circle under iterated
cut-and-paste along $\partial^- L$). The map $F$ extends holomorphically
over the isolated puncture; we claim that it also extends (uniquely, holomorphically) 
over $J$. The resulting extension will be a degree $d$ holomorphic self-map 
from a sphere to itself, which is conjugate to a polynomial.

We now explain how to extend the dynamics of $F$ over $J$ holomorphically.
Let $X$ be the subset of $\Omega$ consisting of points with height $\le t$
where $t$ is less than the height of any critical leaf, and let $Y$ be the
closure of $X-F^{-1}(X)$. Then $Y$ is a
(typically disconnected) compact planar surface with outer boundary 
$\partial^+Y:= \partial X$, and inner boundary 
$\partial^-Y:= \partial Y - \partial^+Y$. The map $F: \partial^-Y \to \partial^+Y$
is a $d$-fold covering map for which every component maps homeomorphically
to its image; thus we may define
$F_1,\cdots,F_d: \partial^+Y \to \partial^- Y$ to be branches of
$F^{-1}$ with disjoint images whose union is $\partial^- Y$.

Suppose that $\partial^+Y=\partial X$ has $e$ components. Let $D$ denote the
disjoint union of $e$ copies of the unit disk $\D$.
We would like to find a holomorphic embedding $\psi:X \to D$, so that 
$J:=D-\psi(X)$ is a Cantor set, and so that $F$ (or, really, its conjugate by
$\psi$) extends holomorphically over $J$.

Let $\Teich$ denote the Teichm\"uller space of holomorphic embeddings 
$\psi: Y \to D$ taking components of $\partial^+Y$ to components of $\partial D$, 
and normalized to take fixed values on three marked points on each component.
We define a {\em skinning map} $\sigma:\Teich \to \Teich$ as follows.
Given $\psi$, we cut out $D - \psi(Y)$ and sew in $d$ copies of 
$D$ by gluing their boundaries to $\psi(\partial^-Y)$ along the identifications
$$\partial D \xrightarrow{\psi^{-1}} \partial^+ Y \xrightarrow{F_j} \partial^- Y
\xrightarrow{\psi} \psi(\partial^-Y)$$
We then uniformize the resulting surface $D'$
to obtain a holomorphic identification $D' \to D$,
and the restriction of this uniformization to $Y$ (which we identify with its
image in $D'$ under $\psi$) is $\sigma(\psi)$.
The skinning map is holomorphic, and therefore distance non-increasing in the
Teichm\"uller metric. In fact it is evidently strictly distance decreasing; 
furthermore, orbits are easily seen to be bounded. Thus $\sigma$ is
uniformly strictly distance decreasing, and  
there is a (unique) fixed point (actually convergence to the fixed
point is easy to see directly by considering moduli of accumulating 
annuli around points of $J$).

By construction, this fixed point
gives the desired embedding of $X$ and extension of $F$.
\end{proof}

We denote by $\DL_d$ the space of degree $d$ dynamical elaminations, thought of as
a subspace of $\EL$. Theorem~\ref{theorem:realization} produces a continuous inverse
to the butcher map $\Phi:\SS_d \to \EL$ called the {\em realization map}
$\Psi:\DL_d \to \SS_d$; in particular, the spaces $\SS_d$ and $\DL_d$ are homeomorphic.

The location of the tips of the critical leaves define local holomorphic coordinates
on $\DL_d$ giving it the structure of a complex manifold. With respect to these
coordinates, $\Phi$ and $\Psi$ are holomorphic; thus $\DL_d$ and $\SS_d$ are isomorphic as
complex manifolds.

\subsection{Squeezing}

There is a free proper $\R$ action on $\DL_d$ which simultaneously multiplies
the heights of the critical leaves by some fixed positive real number $e^t$. 
We call this transformation {\em squeezing}, and refer to the $\R$ action
as the {\em squeezing flow}. 

Since the squeezing flow is (evidently) proper, it gives $\DL_d$ the structure of
a global product:

\begin{corollary}
Each $\DL_d$ is homeomorphic to a product $\DL_d = X_d \times \R$ where $X_d$ is
a real manifold of dimension $2d-3$.
\end{corollary}

For concreteness, we may think of $X_d$ as the subspace of $\DL_d$ where
the largest critical height is equal to $1$.

\subsection{Rotation}

If $P$ is a leaf in $L$, we let $e^{i2\pi t}P$ denote the result of rotating $P$
anticlockwise through $t$, mod leaves of greater height. This makes sense
unless $P$ collides with a leaf of the same height. If $P$ and $Q$ are leaves of
different height, the operations of rotating $P$ and rotating $Q$ commute. 

If $L$ is a dynamical elamination of degree $d$
with distinct critical leaves, let $L_j$ be the critical leaf $C_j$ and its preimages.
Suppose no two critical leaves have heights whose ratio is a power of $d$; we say
$L$ has {\em generic heights}. Then for
a vector $s:=s_1,\cdots s_{d-1}$ of real numbers we can simultaneously 
rotate all the leaves of each $L_j$ of height $h$ through angle $hs_j$, mod leaves of
greater height; since leaves of the same height are all rotated through the same angle, they
never collide and this operation is well-defined. 
Denote the result by $\FF_s L: = \cup_j e^{i2\pi hs_j}L_j$. 

\begin{lemma}[Torus orbits]\label{lemma:rotation_torus}
If $L$ is a degree $d$ dynamical elamination with generic heights $h(C)$, 
then $\FF_s L \in \DL_d$.
Furthermore the orbit map $\R^{d-1} \to \DL_d$ factors through a torus $T_L:=\R^{d-1}/\Gamma_L$
where $\Gamma_L$ is contained in $d^{-n}h(C)^{-1}\Z^{d-1}$ for some $n$. 
\end{lemma}
\begin{proof}
By induction, for each precritical leaf $P$ we have 
$(e^{i\theta}P)^d = e^{i\theta d}P^d$ mod leaves of greater height. Thus $\FF_s L$ is
a degree $d$ dynamical elamination.

For each critical leaf $C_j$ the angles of $C_j$ vary continuously in a component of
$S^1$ mod leaves of greater height. Since the angles of these leaves of greater height
all differ by multiples of $d^{-n}$ for some fixed $n$, the length of this component 
is a multiple of $\R/d^{-n}\Z$. The lemma follows.
\end{proof}

\section{Degree 2}

Our goal in the sequel is to investigate the topology and combinatorics of $\SS_d$.
As a warm-up, and in order to introduce the main ideas in a relatively clean context,
we describe in the next few sections the special cases of degrees $2$, $3$ and $4$.
After developing the theory of the past few sections, the case of degree $2$ 
is almost a triviality.

\begin{theorem}[Douady--Hubbard \cite{Douady_Hubbard_1}]
The space $\SS_2$ is holomorphically equivalent to a punctured disk.
\end{theorem}
\begin{proof}
A degree $2$ dynamical elamination $L$ is generated by a single (necessarily simple)
critical leaf $C$. The tips of $C$ are of the form $\ell(z)$ and $\ell(-z)$ for some
$z \in \E$. Since every other leaf of $L$ has smaller height than $C$, the number $z^2$
is a continuous function of $\DL_2$, and conversely we can recover $C$ and
therefore $L$ from $z^2$. Hence $\DL_2$ is holomorphically isomorphic to the
quotient of $\E$ by $\pm 1$.
\end{proof}

\begin{corollary}
The Mandelbrot Set $\M$ (i.e.\/ the complement of $\SS_2$ in $\C$) is connected.
\end{corollary}

\section{Degree 3}\label{section:degree_3}

\subsection{The Tautological Elamination}\label{subsection:tautological_lamination}

Throughout this section we refer to the {\em angles} of a leaf $P$ of an 
elamination as the arguments of the tips divided by $2\pi$; thus angles take values
in the circle $S^1=\R/\Z$.
 
For some small $\epsilon>0$ and angles $t,s \in S^1$
let $L(t,s)$ be the degree $3$ dynamical
elamination with simple critical leaves $C_1,C_2$ where $C_1$ has height $1$ and
angles $\lbrace t,t+1/3\rbrace$, and $C_2$ has height $1-\epsilon$ and
angles $\lbrace s,s+1/3\rbrace$. Note that this forces $s\in (t+1/3,t+2/3)$.

If we fix $t$ and vary $s$ in $(t+1/3,t+2/3)$, 
then whenever $3^ns$ is equal to $t$ or $t+1/3$, the
leaf $C_2$ collides with a leaf $P$ of $L(t,s)$ which is a depth $n$ preimage of $C_1$.
We define an elamination $\Lambda_T(t)$ whose leaves are the union of the leaves $P^3$
over all $P$ in all $L(t,s)$ of this kind.

\begin{example}
Let $t=0$ and $s=5/9$. Thus $C_1$ has angles $\lbrace 0,1/3\rbrace$ and $C_2$ has
angles $\lbrace 5/9,8/9\rbrace$. There is a unique leaf $P$ with angles
$\lbrace s=5/9, s'\rbrace$ which collides with $C_2$
for which $P^9=C_1$ and neither $P$ nor $P^3$ crosses
$C_1$ or $C_2$ (actually, because $P$ is saturated by $C_2$, it has angles
$\lbrace s=5/9,s',8/9\rbrace$ but we ignore this point, since the tips with angles
$5/9$ and $8/9$ become equal in $P^3$ and it is the leaf $P^3$ that is in $\Lambda_T(0)$). 
The leaf $P^3$ has angles $\lbrace 3s=2/3, 3s'\rbrace$; since
$9s' = 1/3 \mod \Z$, for $P^3$ not to cross $C_1$ or $C_2$ we must have $3s'=7/9$.
Thus, in order for $P$ not to cross $C_1$ or $C_2$ we must have $s'=16/27$. 
See Figure~\ref{tautological_leaf_example}.

\begin{figure}[ht]
\centering
\includegraphics[scale=0.5]{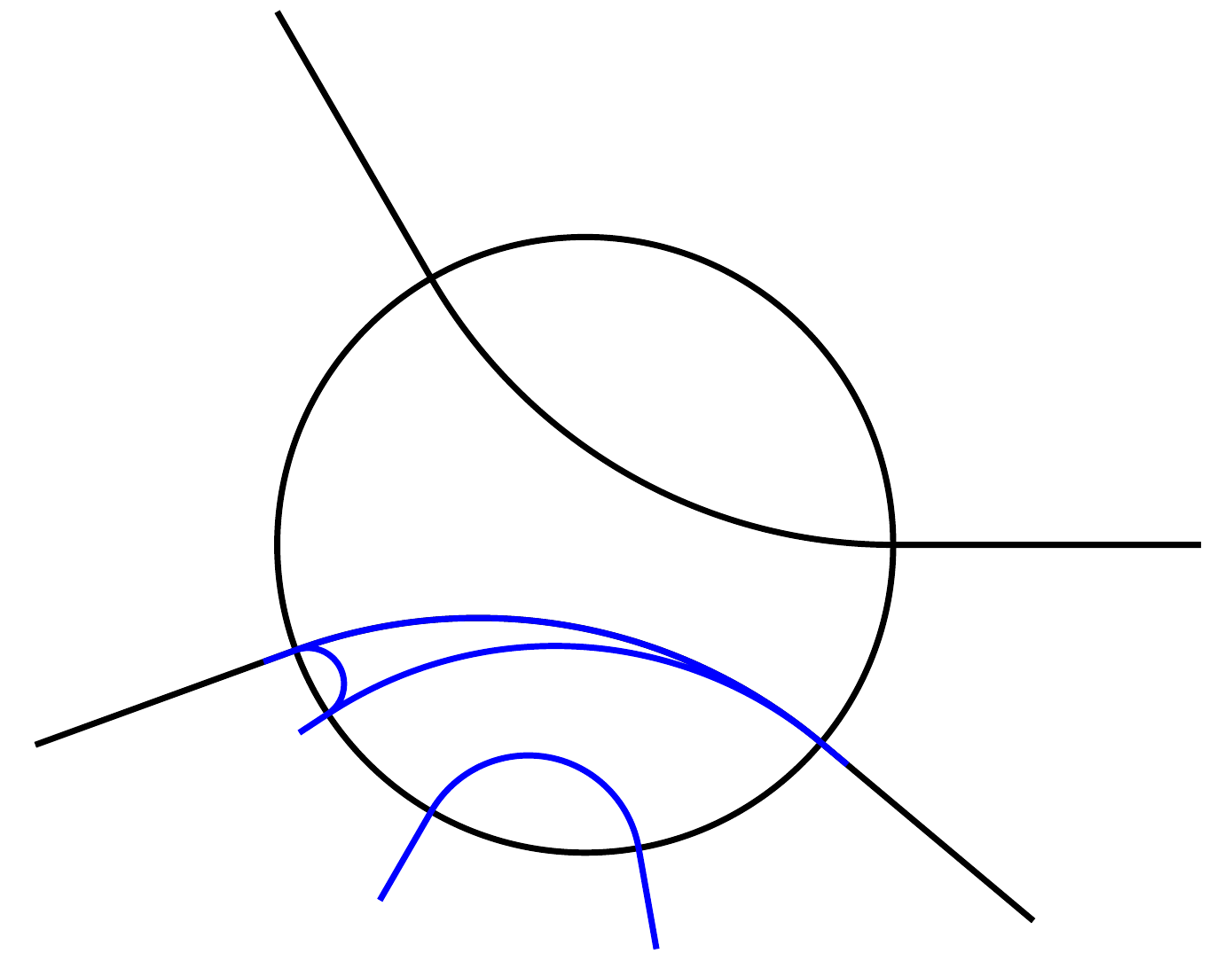} 
\caption{$P$ and $P^3$ (in blue) have angles $\lbrace 5/9,16/27,8/9\rbrace$ and
$\lbrace 2/3,7/9\rbrace$.}\label{tautological_leaf_example}
\end{figure}

The leaf $P^3$ with height $1/3$ and angles $\lbrace 2/3,7/9\rbrace$ is therefore a
leaf of $\Lambda_T(0)$.
\end{example}

\begin{definition}[Tautological Elamination]\label{definition:tautological}
Fix $t \in S^1$. The {\em tautological elamination} $\Lambda_T(t)$ 
is the union of $P^3$ over all leaves $P\in L(t,s)$ in the preimage of $C_1$ over
all values of $s$ at which $C_2 \in L(t,s)$ collides with $P$.
\end{definition}

If $P \in L(t,s)$ is a depth $n$ preimage of $C_1$ that collides with $C_2$,
we refer to its image $P^3 \in \Lambda_T(t)$ as a
{\em depth $(n-1)$ leaf of $\Lambda_T(t)$}.

\begin{proposition}\label{proposition:tautological_lamination}
For all $t$, $\Lambda_T(t)$ is an elamination. Furthermore, $\Lambda_T(t+s)=e^{i2\pi hs}\Lambda_T(t)$
for any $t,s$.
\end{proposition}
\begin{proof}
As we vary $C_2$ fixing its height, the preimages of $C_1$ are occasionally 
pushed over preimages of $C_2$ of greater height. But a depth $1$ preimage $P$ of
$C_1$ has height $1/3$, which is greater than the height of any preimage of
$C_2$, so $P$ is only pushed over $C_2$ itself. Since the angles of $C_2$ differ by
$1/3$, pushing $P$ over $C_2$ does not change its image $P^3$. So we can simply
add $P^3$ to $\Lambda_T(t)$.

Now imagine shrinking the height of $C_2$ to $1/3(1-\epsilon)$ and then varying its
angles again. The depth $1$ preimages of $C_1$ pinch the unit circle into smaller
circles, and $C_2$ is confined to a single component. Since $C_1$ now has
height $<1/3$, the depth $2$ preimages $Q$ of $C_1$ in this component have bigger height
than any preimage of $C_2$, so they stay fixed until they collide with $C_2$,
and we can simply add the $Q^3$ to $\Lambda_T(t)$. In other words: the depth $2$ leaves
of $\Lambda_T(t)$ are the cubes of the depth $2$ preimages of $C_1$ in the component
of $S^1$ pinched along the depth $1$ preimages of $C_1$ containing $C_2$. It follows
that these leaves are disjoint, and do not cross depth $1$ leaves.

Inductively, shrink the height of $C_2$ to $3^{-n}(1-\epsilon)$. It is confined to
a component of $S^1$ pinched along the depth $\le n$ preimages of $C_1$, and as it 
moves around this component, it collides with some depth $(n+1)$ preimages $R$ of
$C_1$ and we add $R^3$ to $\Lambda_T(t)$. It follows (as before) that these leaves
are disjoint and do not cross leaves of depth $\le n$. This proves that $\Lambda_T(t)$
is an elamination.

To see how $\Lambda_T(t)$ varies with $t$, shrink $C_2$ down to the height of a depth
$n$ preimage $P$ it has just collided with. Then rotate $C_1$ and simultaneously
rotate $C_2$ at speed $3^{-n}$ (modulo leaves of greater height)
so that it continues to collide with $P$.
\end{proof}

Figure~\ref{tautological_lamination} depicts subsets of the tautological elaminations up
to depth six associated to $\theta_1=1/12$ in units where the 
unit circle has length 1.

\begin{figure}[htpb]
\centering
\includegraphics[scale=0.23]{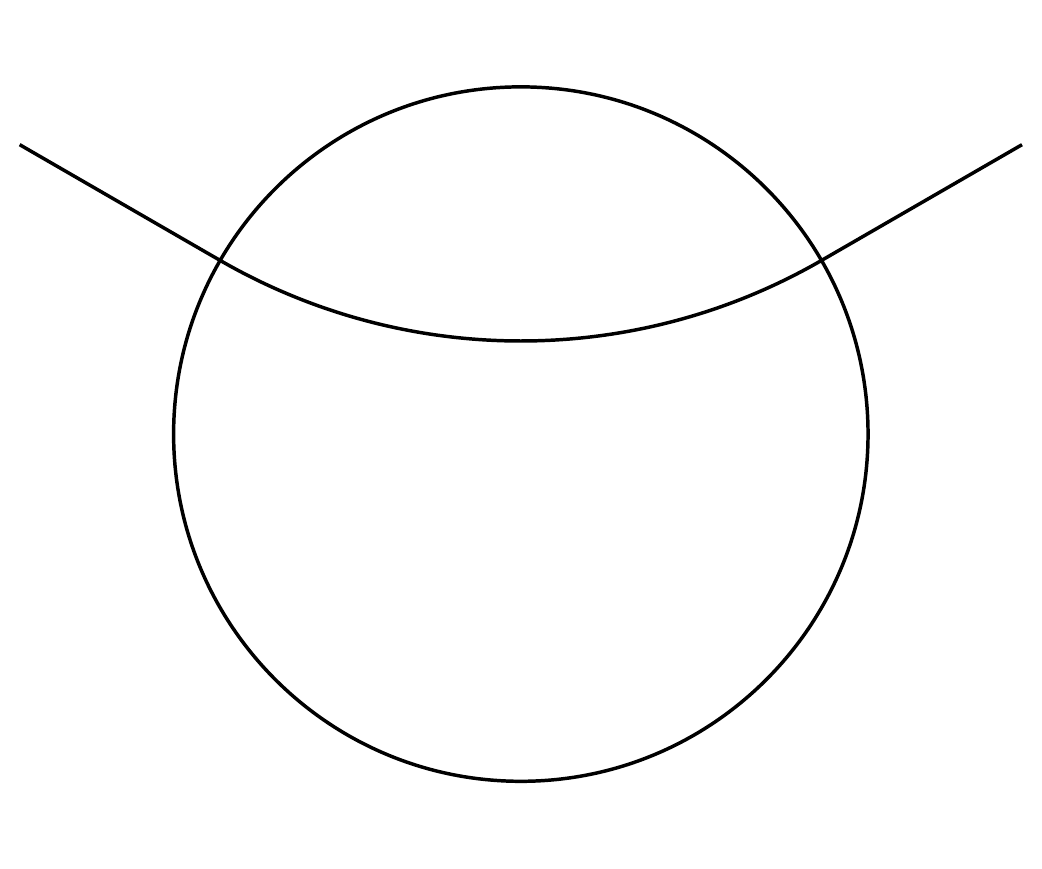} 
\includegraphics[scale=0.23]{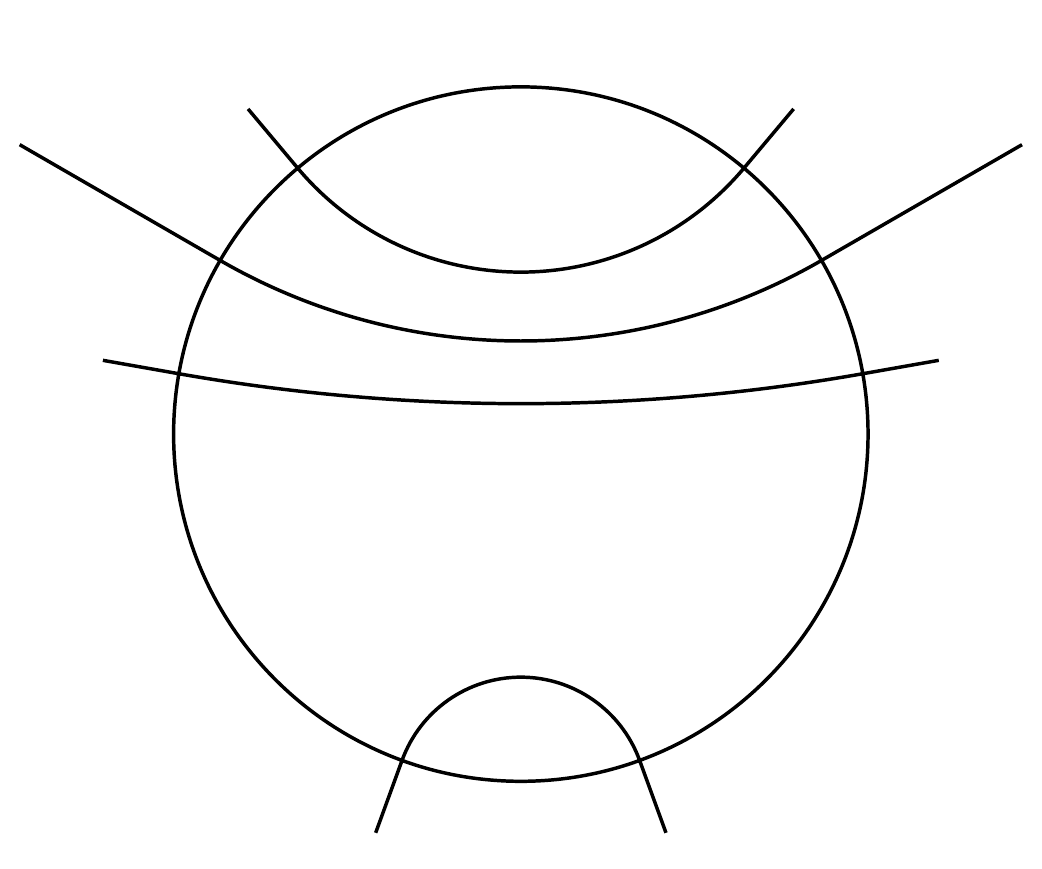} 
\includegraphics[scale=0.23]{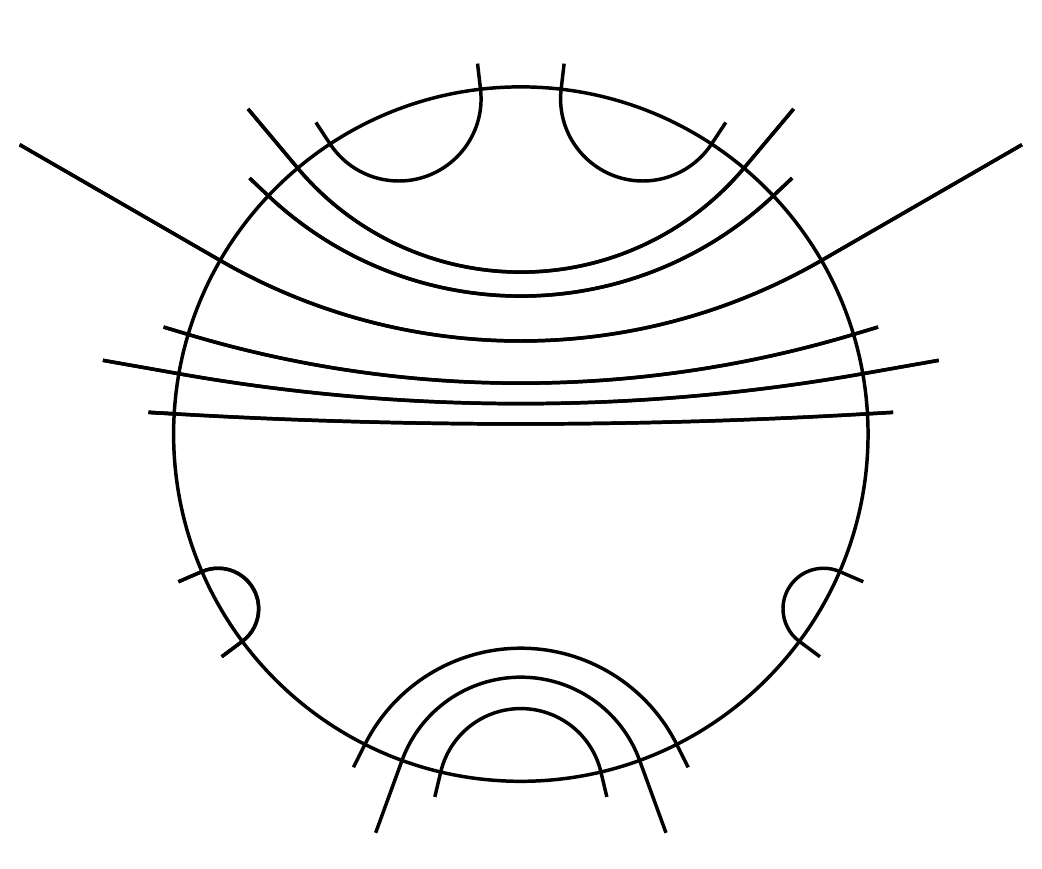} 
\includegraphics[scale=0.23]{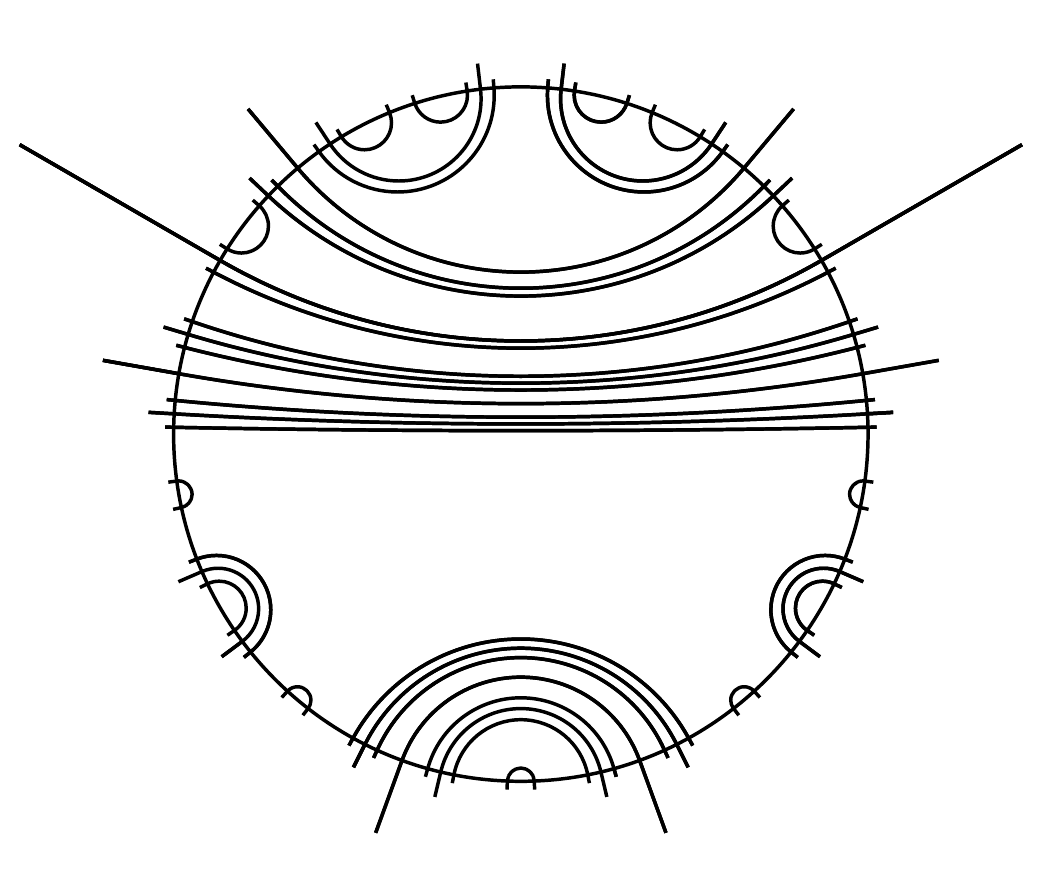} 
\includegraphics[scale=0.23]{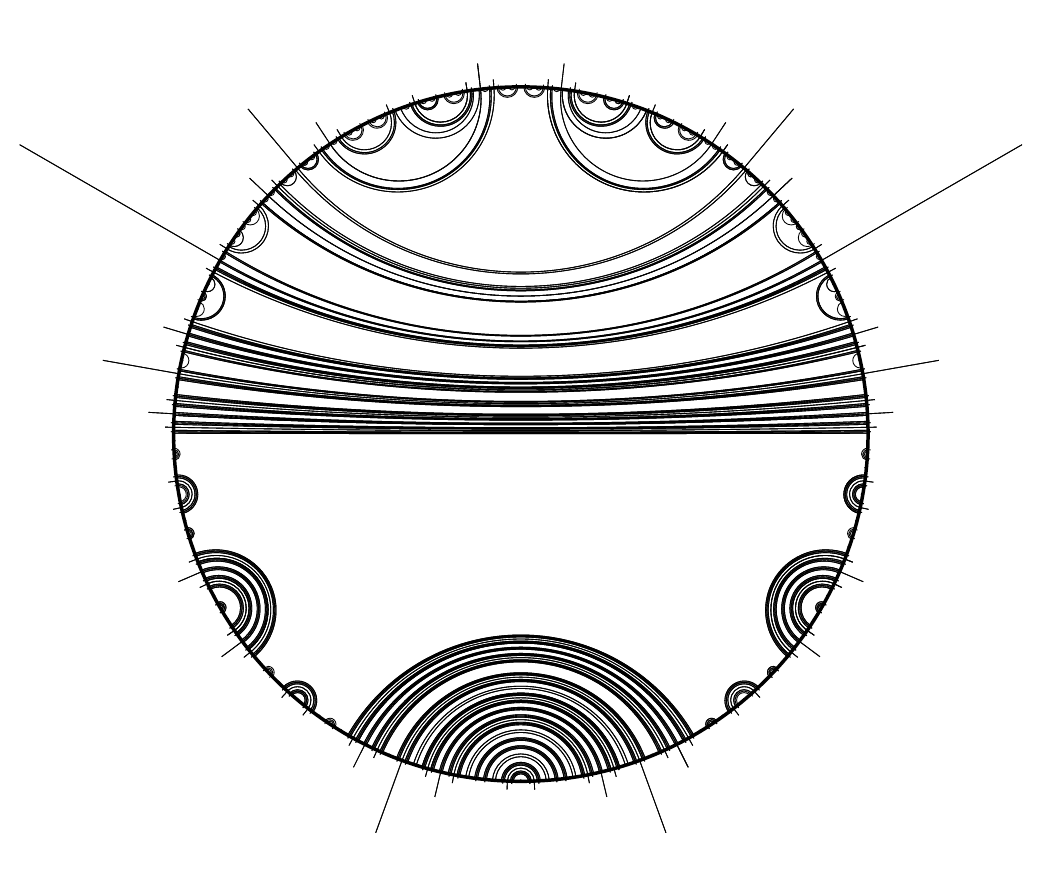} 
\caption{Tautological elaminations $\Lambda_T(1/12)$
to depths 1, 2, 3, 4 and 6}\label{tautological_lamination}
\end{figure}

\subsection{Topology of $\SS_3$}

Let $\Omega(t) = \E \mod \Lambda_T(t)$, and let $D_\infty(t)$ be the subsurface
of $\Omega(t)$ of height $\le 3(1-\epsilon)$. Then $D_\infty(t)$ is a disk minus a Cantor set, and as $t$ varies,
the $D_\infty(t)$ vary by `rotating' the level sets of height $h$ through angle $ht/3$.
By Proposition~\ref{proposition:tautological_lamination} this family of motions for
$t \in [0,1]$ induces a mapping class $\varphi$ of $D(0)$ to itself. The mapping
torus $N_\infty$ of $\varphi$ is the total space of a fiber bundle over $S^1$ whose fiber
over $t$ is $D_\infty(t)$.

Figure~\ref{tautological_pinch} shows a tautological elamination
$\Lambda_T(5/6)$ and the disk $D_\infty$ obtained by pinching it 
(to depth $7$). These pictures were generated by the program {\tt shifty}
\cite{Calegari_shifty} which pinches elaminations recursively one leaf at a time, 
instead of simultaneously pinching all leaves of fixed depth. Thus the picture
of $D_\infty$ is only a combinatorial approximation, and is not conformally accurate.

\begin{figure}[htpb]
\centering
\includegraphics[scale=0.4]{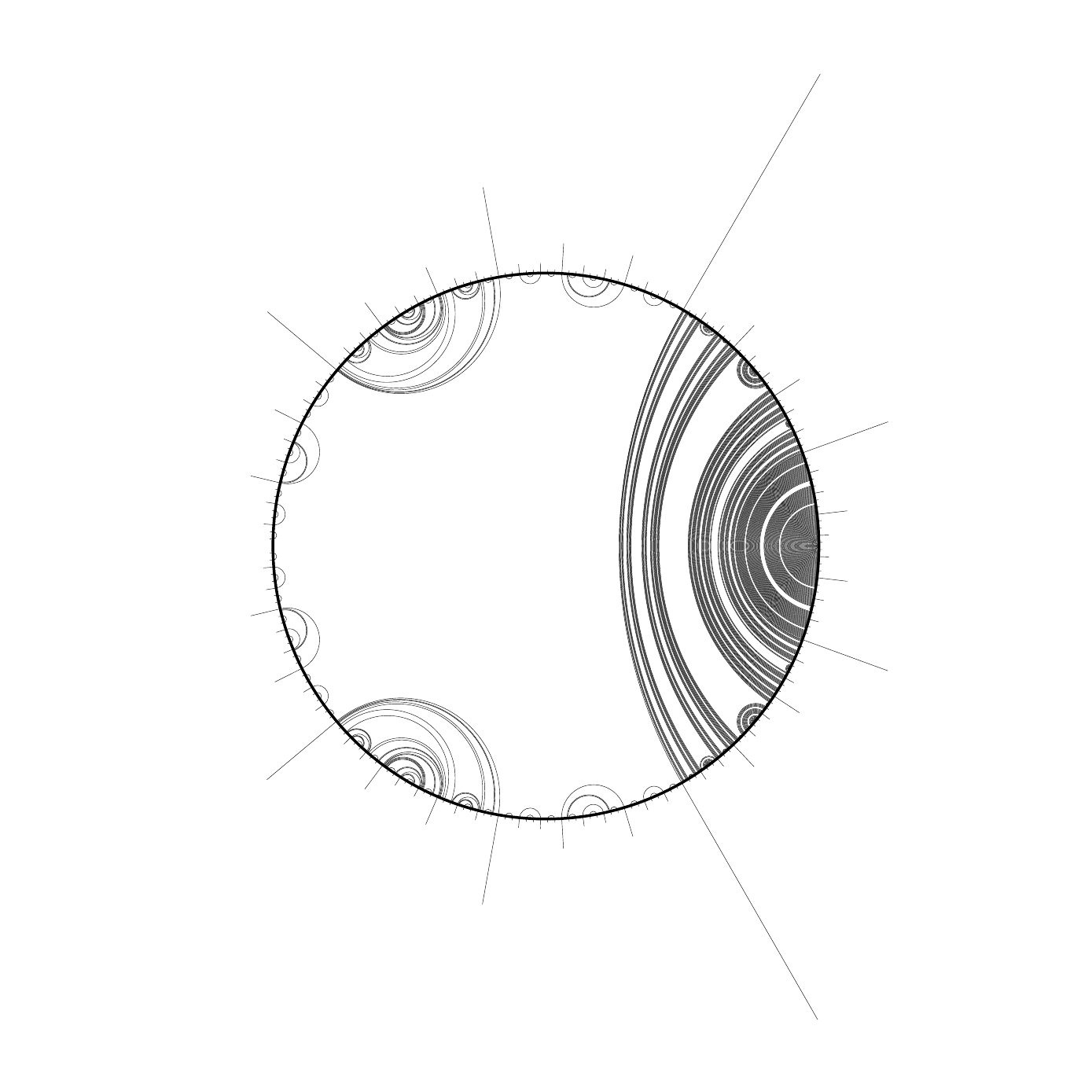} 
\includegraphics[scale=0.4]{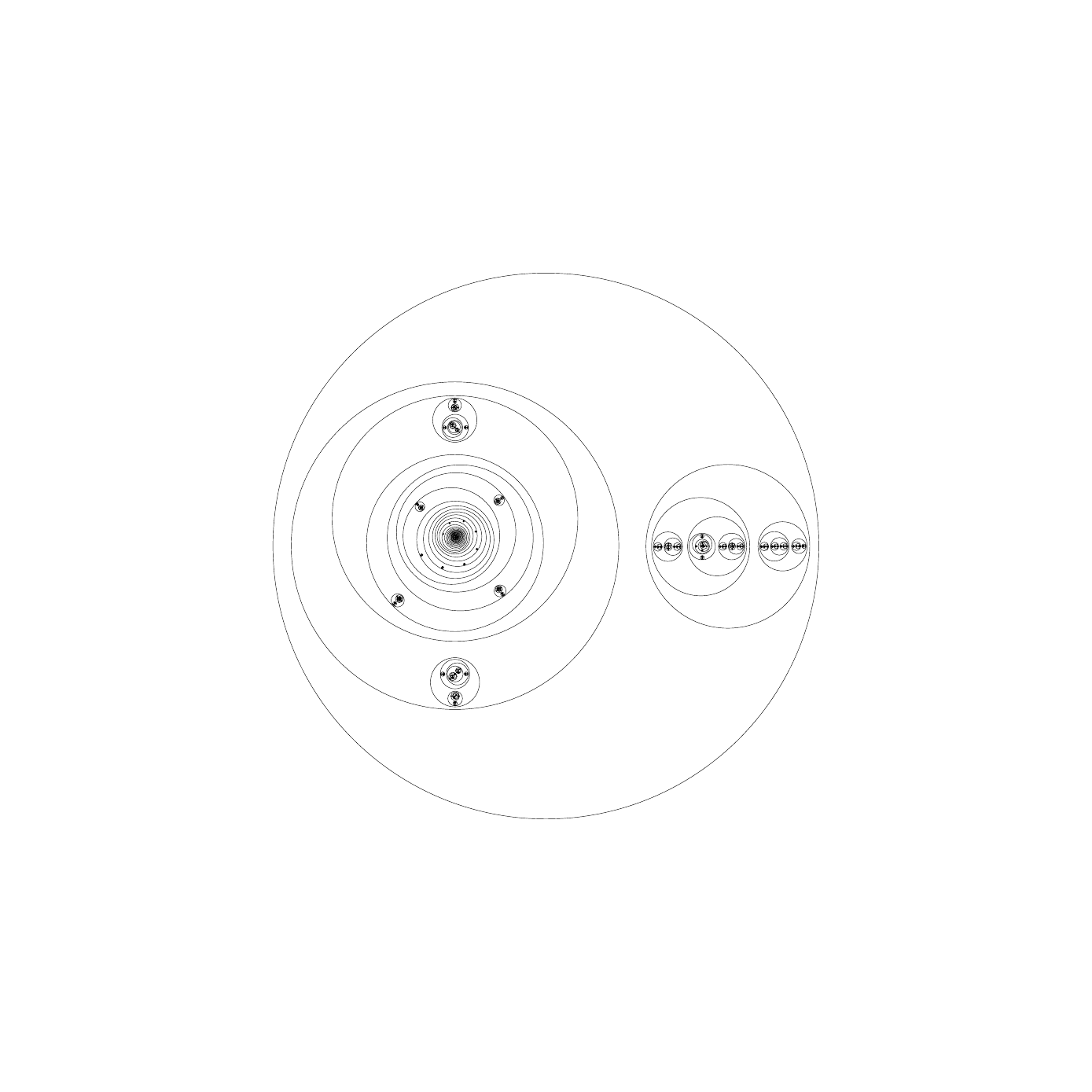}
\caption{Tautological elamination $\Lambda_T(5/6)$ 
and the disk obtained by pinching it}\label{tautological_pinch}
\end{figure}

\begin{theorem}[Topology of $\SS_3$]\label{theorem:S_3_topology}
The space $\SS_3$ is homeomorphic to a product $X_3 \times \R$ where $X_3$ is the
3-manifold obtained from the 3-sphere $S^3$ by drilling out a neighborhood of a 
right-handed trefoil and inserting the mapping torus $N_\infty$, so that the longitude 
intersects the circle $\partial D_\infty(t)$ at angle $t$.
\end{theorem}
\begin{proof}
This follows more or less directly from the definitions. Let's examine the subspace
$Y_3$ of $X_3$ for which $h(C_1)=1$ and $h(C_2)\le(1-\epsilon)$. If we fix $\theta_1$ and the
height $h:=h(C_2)$ then we obtain a (1-dimensional) subspace $\Gamma(\theta_1,h)$
of $Y_3$. Evidently $\Gamma(\theta_1,h)$ is obtained from the circle of
possible $\theta_2$ values $[\theta_1 + 1/3,\theta_1+2/3]/\text{endpoints}$ by
suitable cut and paste. By multiplying angles by $3$ we can identify this space of
$\theta_2$ values with the unit circle $S^1$; so $\Gamma(\theta_1,h)$ is obtained
from $S^1$ by cut and paste. We claim it is precisely equal to 
the result of cut and paste along the leaves of $\Lambda_T(t)$ of height $>h$

To see this, think about a component
$\gamma$ of $\Gamma(\theta_1,h)$; its preimage $\tilde{\gamma}$ in $S^1$ is a 
union of segments. The discontinuities of $\theta_2$ in $\Gamma(\theta_1,h)$ 
occur precisely when $C_2$ is pushed over a precritical leaf of $C_1$ of 
height $>h$; thus the boundary of each component of $S^1 - \tilde{\gamma}$ 
is a precritical leaf $P$ of $C_1$ so that $C_2$ collides with $P$ in 
some dynamical elamination $L(\theta_1,s)$. But then by definition $P^3$ 
is a leaf of $\Lambda_T(\theta_1)$, and all leaves of $\Lambda_T(\theta_1)$ 
arise this way. This proves the claim, and shows that $Y_3$ is homeomorphic 
to $N_\infty$.

It remains to show that $X_3-Y_3$ is homeomorphic to the complement of the right handed
trefoil. For each $h \in (1/3,1)$ the slice of $X_3$ for which $h(C_2)=h$ is just
a torus $T$, with coordinates $\theta_1 \in S^1$ and 
$\theta_2 \in [\theta_1 + 1/3,\theta_1+2/3]/\text{endpoints}$. When $h=1$ we can
no longer distinguish $C_1$ and $C_2$, so this torus is quotiented out by the
involution switching $\theta_1$ and $\theta_2$ coordinates; the
quotient is a circle bundle over an interval with orbifold endpoints of
orders $2$ and $3$ --- see Figure~\ref{quotient_torus}. 
Thus $X_3-Y_3$ is a circle bundle over a disk with two orbifold points, 
one of order $2$ and one of order $3$; this is the standard 
Seifert fibered structure on $S^3 - \text{trefoil}$.
\end{proof}

\begin{figure}[htpb]
\centering
\includegraphics[scale=1]{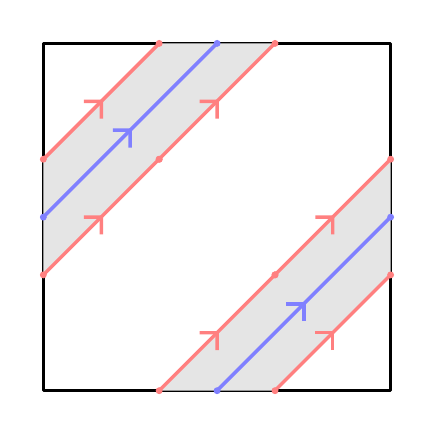} 
\caption{Quotient of the torus $T$ by the involution switching $\theta_1$ and $\theta_2$
is a circle bundle over an interval with orbifold endpoints of orders $2$ and $3$.}\label{quotient_torus}
\end{figure}

\subsection{Geometry and topology of $X_3$}

Let $\Lambda_T(\theta_1,n)$ denote the finite elamination consisting of the
leaves of $\Lambda_T(\theta_1)$ of depth $\le n$ (i.e.\/ they correspond in the
construction of the tautological elamination to depth $n$ preimages of $C_1$).

Let $\Omega_n(\theta_1)$ be the Riemann surface obtained by pinching
$\Lambda_T(\theta_1,n)$ and let $D_n(\theta_1)$ be the subsurface of height
$3(1-\epsilon)$. Then each $D_{n+1}(\theta_1)$ is obtained by pinching 
$D_n(\theta_1)$ along the depth $(n+1)$ leaves, and we can think of 
$D_\infty(\theta_1)$ as the limit. Likewise we can define mapping tori 
$N_n$ which are $D_n(\theta_1)$ bundles over the $\theta_1$ circle $S^1$.

Let $M_n$ denote the result of inserting $N_n$ into the right-handed trefoil
complement in $S^3$. Then $M_n$ is a link complement, $S^3 - K_n$ where $K_0$
is the trefoil itself and each $K_{n+1}$ is obtained from $K_n$ by (a rather
simple) satellite of its components. The limit $K_\infty = S^3 - X_3$ is a 
Cantor set bundle over $S^1$; one sometimes calls such objects {\em Solenoids}.

We now state and prove two theorems, which describe $\SS_3$ in geometric
resp. topological terms. The geometric statement is that $\SS_3$ is homotopic
to a {\em locally $\CAT(0)$ 2-complex}. This means a 2-dimensional
CW complex (in the usual sense) with a path metric of non-positive curvature; 
see e.g.\/ \cite{Bridson_Haefliger} for an introduction to the theory of $\CAT(0)$ spaces. 

The most important corollary of this structure for us is that a locally
$\CAT(0)$ complex is a $K(\pi,1)$; the proof is a generalization 
of the usual proof of the Cartan--Hadamard theorem for complete Riemannian 
manifolds of nonpositive curvature 
(which are themselves examples of locally $\CAT(0)$ spaces). Thus (for example)
$\pi_1(\SS_3)$ is torsion free, and has vanishing homology with any
coefficients in dimension greater than 2.

\begin{theorem}[$\CAT(0)$ 2-complex]\label{theorem:S_3_CAT0_2complex}
$\SS_3$ is a $K(\pi,1)$ with the homotopy type of a locally
$\CAT(0)$ 2-complex.
\end{theorem}
\begin{proof}
Up to homotopy, we can take $M_0$ to be the {\em spine} of the trefoil complement;
this is the mapping torus of a theta graph by an order three isometry that
permutes the edges by a cyclic symmetry. It can be thickened slightly to $M_0$ by
gluing on a metric product (flat) torus times interval. Each $M_n$ has boundary a
union of totally geodesic flat tori, and each $M_{n+1}$ is obtained by gluing
a flat annulus whose boundary components are parallel geodesics in $\partial M_n$
(circlewise, the endpoints of a leaf of the tautological elamination of depth
$(n+1)$) and then gluing a flat torus times interval on each resulting
boundary component to thicken. The union is homeomorphic to $X_3$. 

Simply gluing the spines at each stage without thickening gives a homotopic complex
which is evidently $\CAT(0)$.
\end{proof}

\begin{corollary}
$\pi_1(\SS_3)$ is torsion-free, and homology with any coefficients vanishes in
dimension greater than 2.
\end{corollary}

The topological statement is that $X_3$ is homeomorphic to a Solenoid complement
of a particularly simple kind: one obtained as an infinite increasing union of 
iterated cables.

\begin{theorem}[Link complement]\label{theorem:S_3_link_complement}
The degree $3$ shift locus $\SS_3$ is homeomorphic to $X_3 \times \R$
where $X_3$ is $S^3$ minus a Solenoid $K_\infty$ obtained as a limit of a sequence
of links $K_n$ where
\begin{enumerate}
\item{$K_0$ is the right-handed trefoil; and}
\item{Each component $\alpha$ of $K_n$ gives rise to new components
$\alpha_0 \cup \alpha_c$ of $K_{n+1}$, where $\alpha_0$ is the core of a neighborhood
of $\alpha$ (i.e.\/ we can think of it just as $\alpha$ itself) and 
$\alpha_c$ is a finite collection of $(p_\alpha,q_\alpha)$ cables of $\alpha_0$, 
for suitable $p_\alpha,q_\alpha$.}
\end{enumerate}
\end{theorem}
\begin{proof}
The only thing to prove is the second bullet point. Let $\alpha$ be a component
of $K_n$. The boundary of a tubular neighborhood of $\alpha$ is the mapping
torus of a finite collection of boundary circles of $D_n(0)$ which are 
permuted by the monodromy $\varphi$. Let $m$ be the least power of $\varphi$ that
takes one such boundary component $\gamma \subset \partial^- D_n(0)$ to itself.
Then $\varphi^m$ acts on $\gamma$ by rotation through $2\pi p_\alpha/q_\alpha$.

The depth $(n+1)$ leaves of $\Lambda_T$ on the component $\gamma$ form a 
finite elamination permuted by $\varphi^m$. Think of this as determinining 
a finite geodesic lamination of $\D$. The complementary components are in bijection
with the components $\gamma_j$ of $\partial^- D_{n+1}(0)$ obtained by pinching 
$\gamma$, and we must understand how $\varphi^m$ acts on
them. A finite order rotation of $\D$ has a unique fixed point --- the center. So
there is a unique component $\gamma_0$ invariant under $\varphi^m$, and all the
other components are freely permuted with period $q_\alpha$. Evidently under
taking mapping tori $\gamma_0$ is associated to the core $\alpha_0$ 
and the other $\gamma_j$ are associated to components $\alpha_c$ which are all
$(p_\alpha,q_\alpha)$ cables of $\alpha_0$.
\end{proof}

\begin{corollary}[Homology of $\SS_3$]
$H_1$ and $H_2$ of $\SS_3$ (and of $\pi_1(\SS_3)$) is free abelian on 
countably infinitely many generators. $H_0=\Z$ and $H_n=0$ for all $n>2$.
\end{corollary}

In fact, it is possible to get more precise information about the denominators
$q_\alpha$, and in fact we are able to show:

\begin{theorem}[Powers of 2]\label{theorem:powers_of_2}
The orbit lengths under $\varphi$ of the cuffs of $D_n$ 
(and hence all denominators $q_\alpha$ in Theorem~\ref{theorem:S_3_link_complement}) 
are powers of 2.
\end{theorem}
In fact, the proof of Theorem~\ref{theorem:powers_of_2} 
goes via arithmetic, and will be given in \S~\ref{section:sausages};
technically, the proof is a consequence of Theorem~\ref{theorem:monkey_moduli} and
Example~\ref{example:degree_3_sausages}. We do not actually know
a direct combinatorial proof of this theorem in terms of the combinatorics of
the tautological elamination, and believe it would be worthwhile to try to find one.
We explore the combinatorics of the tautological elamination further in
\S~\ref{subsection:sausage_tautological}.

The tautological elamination has exactly $3^{n-1}$ leaves of depth $n$ and therefore
$(3^n-1)/2$ leaves of depth $\le n$. It follows that $D_n$ is a disk with
$(3^n+1)/2$ holes. However, the monodromy $\varphi$ permutes these nontrivially,
and $K_n$ has one component for each orbit.

The links $K_n$ have $1,2,5,11$ components for $n=0,1,2,3$, though the degrees with
which these components wrap around the cores of their parents are quite complicated. 
Thickened neighborhoods of $K_n$ for $n=0,1,2$ are depicted in Figure~\ref{trefoil_cables}.

\begin{figure}[htpb]
\centering
\includegraphics[scale=0.1]{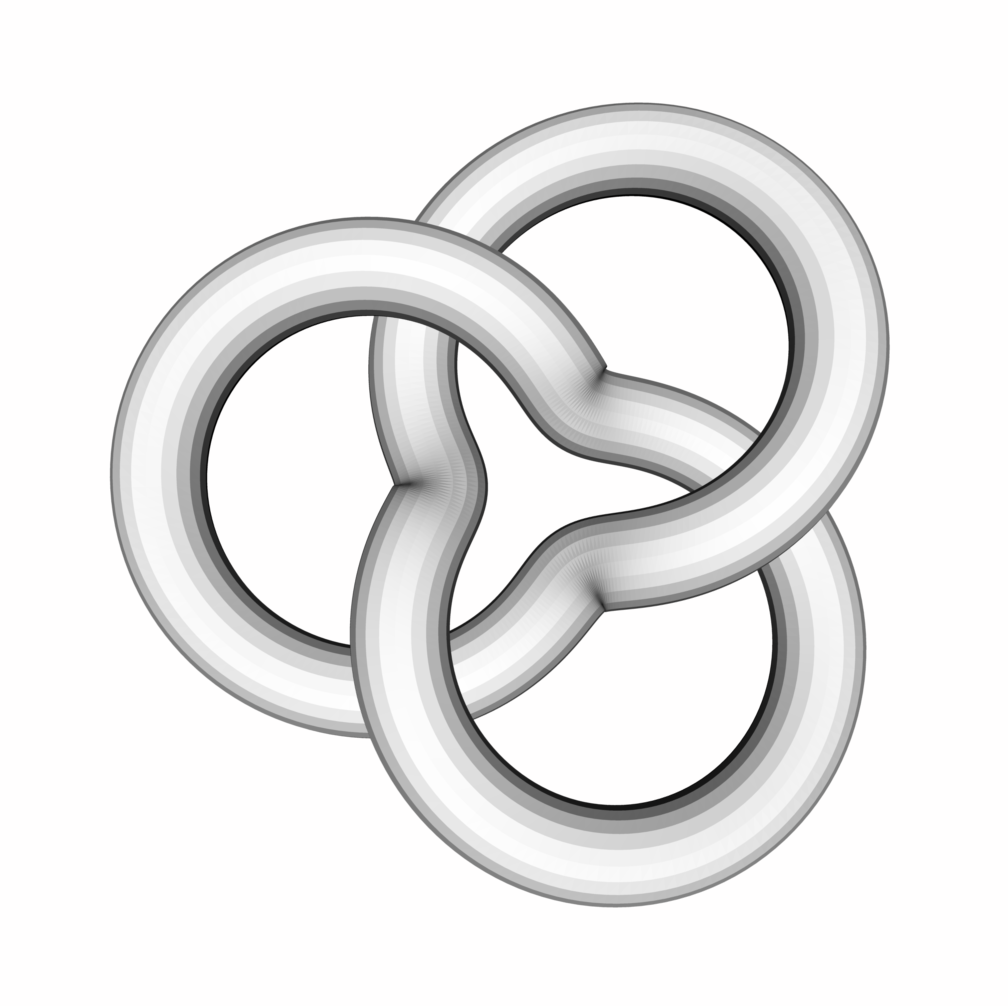} 
\includegraphics[scale=0.1]{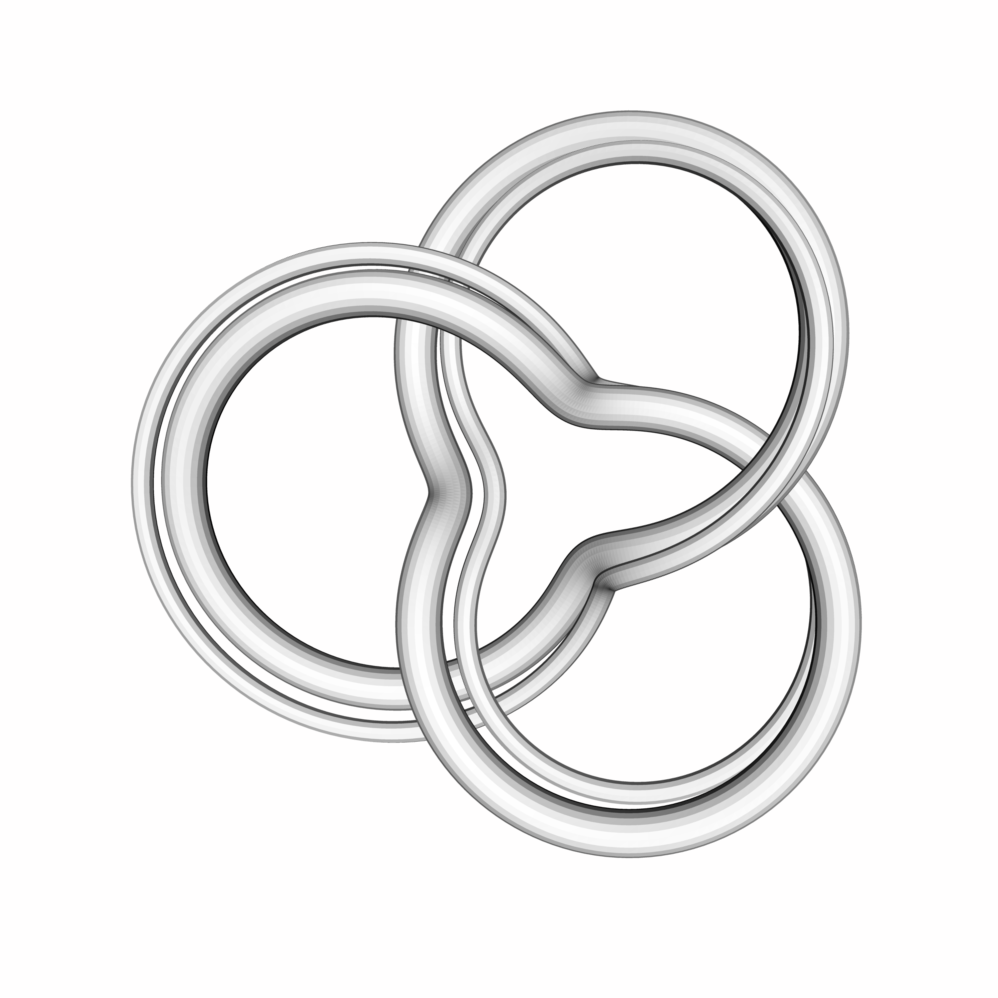} 
\includegraphics[scale=0.1]{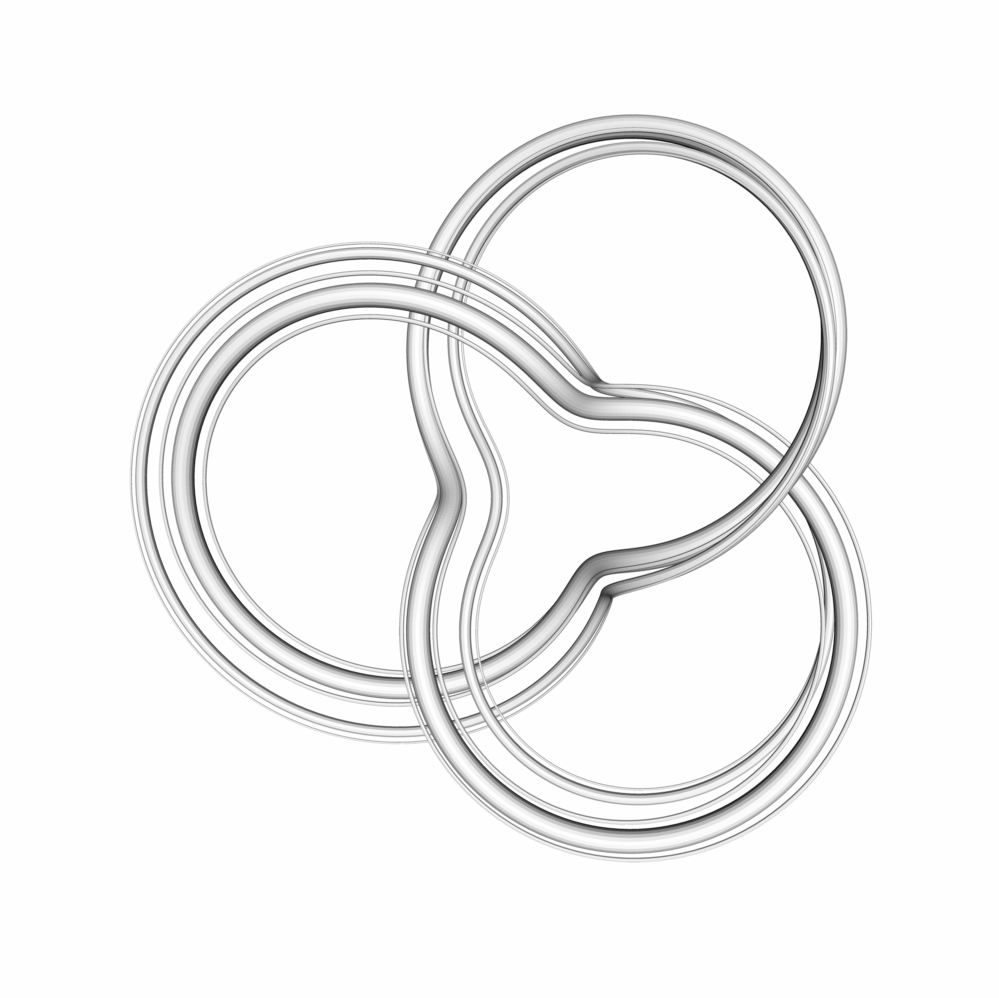} 
\caption{Thickened neighborhoods of $K_j$ for $j=0,1,2$. $X_3$ is 
homeomorphic to $S^3-K_\infty$}\label{trefoil_cables}
\end{figure}

\section{Degree 4 and above}

\subsection{Weyl chamber}

As in the case of degree $3$, we set $\SS_4 = X_4 \times \R$ where $X_4$
is the quotient of $\SS_4$ by the orbits of the squeezing flow.

Order the critical heights with multiplicity so that $h_1 \ge h_2 \ge h_3$
and define a map $\rho:X_4 \to \R^3$ with coordinates $t_j:=-\log_4 h_j$.
If we identify $X_4$ with the subspace for which $h_1=1$ then $t_1=0$
and the image of $\rho$ is the subset of
$(t_2,t_3) \in \R^2$ with $0 \le t_2 \le t_3$. Another normalization is to set
$\sum t_j = 0$ in which case the image of $\rho$ may be identified with the
Weyl chamber $W$ associated to the root system $A_2$.

Within this chamber we have a further stratification. Define 
$t_{ij}:=t_i - t_j$ and refer to the level sets $t_{ij}=n\in \Z$ as {\em walls}.
The walls define a cell decomposition $\tau$ of $W$ into right angled triangles
with dual cell decomposition $\tau'$. 

We shall describe a natural partition of $X_4$ into manifolds with corners
$X_4(v)$, for vertices $v$ of $\tau$, where $X_4(v)$ is defined to be the
preimage under $\rho$ of the cell of $\tau'$ dual to $v$. 
These submanifolds are typically disconnected,
and the way their components are glued up in $X_4$ will give $X_4$ the
structure of a {\em complex of spaces over a contractible $\tilde{A}_2$ building}.

\subsection{Two partitions}

Let's suppose critical leaves are simple, and we label them $C_j$
compatibly with the ordering on heights.

There are two combinatorially distinct ways for $C_1$
to sit in the circle: the angles of the segments are either antipodal, or they are distance
$1/4$ apart (remember we are working in units where the circle has total length $1$).
When $h(C_1)$ is strictly larger than the other $h(C_j)$ the leaf $C_1$ is the unique
leaf of greatest height. Thus the difference of the angles is locally constant;
it follows that the subset of $X_4$ where $h(C_2)<1$ is disconnected. In fact,
it is easy to see it has exactly two components according to the placement of $C_1$.

Where $C_1$ is an antipodal leaf, it pinches the unit circle into two circles of
length $1/2$, each bisected by one of $C_2$ and $C_3$. The restriction of the
dynamical elamination in each of each of these length $1/2$ circles is 
symmetric under the antipodal map. 

When $C_1$ is not antipodal, it pinches the unit
circle into circles of length $1/4$ and $3/4$, with $C_2$ and $C_3$ both contained in
the longer circle. The leaf $C_2$ pinches this circle into circles of length $1/2$
and $1/4$, and $C_3$ divides the length $1/2$ circle antipodally.

\subsection{Monkey prisms, monkey turnovers}

Let's fix a generic $(t_2,t_3)$ in the interior of $W$, 
so that none of $t_2,t_3,t_3-t_2$ are integers.
Denote the fiber of $\rho$ over $(t_2,t_3)$ by $T(t_2,t_3)$. These fibers are
disjoint union of 3-tori, orbits of the $\R^3$ action $\FF_s$ on $\DL_4$ described
in Lemma~\ref{lemma:rotation_torus}. These tori piece together to form a product
throughout each open triangle of $\tau$. We let $\theta_j$ (taking values in $\R^3$ mod
a suitable lattice) denote angle coordinates on one of these tori.

As we pass through a wall where some $t_{ij}\in \N$, circle factors in these tori
pinch as follows. The angle coordinates $\theta$ and the log height coordinates $t$
determine a dynamical elamination. When $t_{ij}=n$ the circle parameterized by
$\theta_i$ is pinched along the precritical leaves of $C_j$ of depth $n$. As we move
around in the fiber, the dynamical elamination varies by a rotation, so the way
in which the $\theta_i$ circle pinches depends only on which component we are in, and
the value of the local coordinates $\theta_j$ with $j<i$. In other words, the
structure locally is that of a certain kind of iterated fiber bundle called a
{\em monkey bundle}.
 
Recall from Definition~\ref{definition:monkey_pants} the terms monkey pants and 
monkey Morse functions.

\begin{definition}[Monkey bundle]\label{definition:monkey_bundle}
A {\em monkey bundle} of {\em order $n$} consists of the following data:
\begin{enumerate}
\item{A finite sequence of fiber bundles $\Omega_2 \to E_2 \to S^1$ and
$\Omega_j \to E_j \to E_{j-1}$ for $3 \le j \le n$ where each $\Omega_j$ is
a monkey pants;}
\item{a map $\pi_j:E_j \to [0,1]$ whose restriction to each $\Omega_j$ fiber
is monkey Morse; and such that}
\item{if $E:=E_n$ is the total space, and $\pi:E \to [0,1]^{n-1}$ denotes the map whose
factors restrict to $\pi_j$ on each $E_j$, then for each $j$ the image of the
critical points in the $\Omega_j$ fibers is a collection of affine hyperplanes.}
\end{enumerate}
\end{definition}

The cube $[0,1]^{n-1}$ together with the hyperplanes which are the images of fiberwise
critical points under $\pi$ should be thought of as a {\em graphic} in the sense
of Cerf theory; see e.g.\/ \cite{Cerf}. We say that a curve in $[0,1]^{n-1}$
crosses a hyperplane of the graphic {\em positively} if it corresponds to the
positive direction in the factor $\pi_j:E_j \to [0,1]$ to which the hyperplane is
associated.

\begin{definition}[Monkey prism; monkey turnover]\label{definition:monkey_turnover}
Suppose $E$ is a monkey bundle with projection $\pi:E \to [0,1]^{n-1}$.
Suppose $\Delta \subset [0,1]^{n-1}$ is a convex polyhedron for which there is a 
vertex $v \in \Delta$ so that the ray from $v$ to every other point in $\Delta$
crosses the graphic in the positive direction. Then we call
$P:=\pi^{-1}(\Delta)$ a {\em monkey prism}.

Suppose $\pi:P \to \Delta$ is a monkey prism, 
and some collection of finite groups act on 
some boundary strata of $P$ preserving $\pi$. Then the quotient space $Q$
of $P$ together with the data of its induced projection to $\Delta$ is
called a {\em monkey turnover}. 
\end{definition}

\begin{lemma}[Prism is $K(\pi,1)$]\label{lemma:prism_K_pi_1}
A monkey prism of order $n$ is a $K(\pi,1)$ with the homotopy type of an
$n$-complex. A monkey turnover of order $n$ has the homotopy type of an
$n$-complex.
\end{lemma}
\begin{proof}
A monkey pants is homotopic to a graph, and iterated fibrations of $K(\pi,1)$s are
$K(\pi,1)$s. Thus a monkey bundle is a $K(\pi,1)$ with the homotopy type of an
$n$-complex. 

The universal cover $\tilde{E}$ of a monkey bundle $E$ is a (noncompact) 
manifold with corners, and interior homeomorphic to a product 
$\R^2 \times \cdots \times \R^2 \times \R$ where each 
$\R^2$ factor has a singular foliation with leaf space an oriented tree. 

If $F \subset E$ is a monkey prism associated to a polyhedron
$\Delta \subset [0,1]^{n-1}$ then the preimage 
$\tilde{F} \subset \tilde{E}$ is bounded in each $\R^2$ factor by a collection of
lines of the foliation, and is homeomorphic to a disjoint union of $\R^2$s. 
As we move along a straight ray in $\Delta$ from the distinguished vertex
we might cross hyperplanes of the graphic,
but by hypothesis we only cross in the positive direction. As we cross a
hyperplane, the part of $\tilde{F}$ in some $\R^2$ fibers splits apart, but
pieces can never recombine; thus $\tilde{F}$ is homeomorphic to $\R^{2n-1}$ 
so that $F$ is also a $K(\pi,1)$ with the homotopy type of an $n$-complex.

Since orbifolding is compatible with $\pi$, a monkey turnover also has the
homotopy type of an $n$-complex.
\end{proof}

From the description of the fibers of $\rho$ and how they pinch as we cross a wall,
the following is immediate:
\begin{lemma}
Let $\Delta$ be a cell of the dual cellulation $\tau'$. Then 
$\rho^{-1}(\Delta)$ is a disjoint union of monkey prisms and monkey turnovers 
with respect to the map $\rho$. 
\end{lemma}

Figure~\ref{quotient_torus} is a simple example of the way a fiber 
can be quotiented in a monkey turnover.

There does not seem to be any obvious reason why monkey turnovers in generality
should be $K(\pi,1)$s. However it will turn out that the turnovers that occur
in the partition of $X_4$ {\em are} $K(\pi,1)$s. The reason for this is subtle,
and only proved in \S~\ref{section:sausages}.

There is another natural cellulation $\kappa$ of $W$ associated to the subset of walls
of the form $t_{i1}\in \N$; i.e.\/ the walls of the integer lattice in $\R^2$.
They decompose $W$ into squares and right-angled triangles. Let $\kappa'$ be
the dual cellulation; the cells of $\kappa'$ are triangles, squares and rectangles,
and the cells of $\kappa'$ are in bijection with the cells of $\tau'$.
Since $\tau$ and $\kappa$ have the same set of vertices, there is a bijection
between the top dimensional cells of $\tau'$ and '$\kappa'$.

In the sequel it will be convenient to compare the monkey prisms and turnovers
associated to $\tau'$ with those associated to $\kappa'$.

\begin{lemma}[Equivalent Cells]\label{lemma:equivalent_cells}
Let $K$ and $T$ be cells of the cellulations $\kappa'$ and $\tau'$ associated to
a vertex $v$. Then the components of $\rho^{-1}(K)$ and of $\rho^{-1}(T)$ are
homeomorphic, and are isotopic inside $X_4$.
\end{lemma}
\begin{proof}
There is an isotopy of the frontiers of the cells from one to the other 
which never introduces any new tangency with the graphic. Since fibers are 
arranged in a product structure away from the graphic, the lemma follows. 
\end{proof}

The prisms and turnovers associated to cells of $\kappa'$ are naturally
homeomorphic to the {\em moduli spaces} introduced in \S~\ref{subsection:moduli_spaces}.

\subsection{$K(\pi,1)$}\label{subsection:K_pi_1}

Decompose $W$ into cells dual to the cellulation by walls; note that
typical cells (those dual to interior vertices of $W$) are hexagons. 
The preimage under $\rho$
of each of these cells is a disjoint union of monkey prisms and monkey
turnovers, and the walls in each cell are the graphic. 
Thus $X_4$ is a {\em complex of spaces} in the sense of Corson \cite{Corson}.
The associated complex is built from copies of cells of $\tau$ according to the
pattern of inclusion of connected components; thus it is an example of an
{\em $\tilde{A}_2$ building}, which comes with an immersion to $W$. 
See e.g.\/ Brown \cite{Brown} for an introduction to the theory of buildings.

\begin{theorem}[Complex of spaces]\label{theorem:degree_4_complex_of_spaces}
$X_4$ is a complex of monkey prisms and monkey turnovers 
over a contractible $\tilde{A}_2$ building $B$.
\end{theorem}
\begin{proof}
The direction of pinching is transverse to the walls, so there is a unique path
in the building from every point to the origin projecting to a ray in $W$. 
\end{proof}

In retrospect, the inductive picture of $X_3$ we obtained in \S~\ref{section:degree_3}
as an infinite union of knot and link complements, exhibits it as a complex
of monkey prisms and monkey turnovers (actually, only one monkey turnover) 
over a contractible $\tilde{A}_1$ building (i.e.\/ a tree).

The next theorem is the analog in degree $3$ of 
Theorem~\ref{theorem:S_3_CAT0_2complex}.
 
\begin{theorem}[$K(\pi,1)$]\label{theorem:S_4_K_pi_1}
$\SS_4$ is a $K(\pi,1)$ with the homotopy type of a 3-complex.
\end{theorem}

We have already seen that the monkey prisms (and consequently also monkey turnovers)
in $X_4$ have the homotopy type of 3-complexes. The same is therefore true of $X_4$. 

$X_4$ is assembled from monkey prisms and monkey turnovers associated to the
vertices of $B$. The edges and triangles are associated to lower dimensional
monkey prisms and turnovers included as facets in the boundary. 
The monkey prisms and their boundary strata are all $K(\pi,1)$s by 
Lemma~\ref{lemma:prism_K_pi_1}, and the inclusions of boundary strata are
evidently injective at the level of $\pi_1$. It remains to show that the same
holds for the monkey turnovers.

We defer the proof of this to \S~\ref{section:sausages}, but for the moment we 
give some examples to underline how complicated the monkey turnovers can be.

\begin{example}[$K(B_4,1)$]
The turnover associated to the vertex $(0,0)$ homotopy retracts onto the fiber
$\rho^{-1}(0,0)$. This is the (3 real dimensional) 
configuration space of degree 4 dynamical elaminations
with all critical leaves of height $1$. This turns out to be a spine for 
the configuration space of $4$ distinct unordered points in $\C$; i.e.\/ it is
a $K(B_4,1)$ (an analogous statement holds in every degree). 
There are several ways to see this; one elegant method is due to Thurston,
and explained in \cite{Thurston_laminations}. We shall see a quite different
and completely transparent demonstration of this fact in \S~\ref{section:sausages}.
\end{example}

\begin{example}[Star of David]\label{example:lamination_star_of_david}
There are two monkey turnovers associated to the vertex $(1,1)$ of $\tau$ in $W$,
corresponding to the two combinatorially distinct ways for
$C_1$ to sit in $S^1$.

When $C_1$ is antipodal, the leaves $C_2$ and $C_3$ sit on either side and do not 
interact with each other. For each fixed value of $C_1$ the other two leaves 
vary as a product $P \times P$ of pairs of pants. Monodromy around the $C_1$
circle switches the two factors by an involution.

\begin{figure}[htpb]
\centering
\includegraphics[scale=0.375]{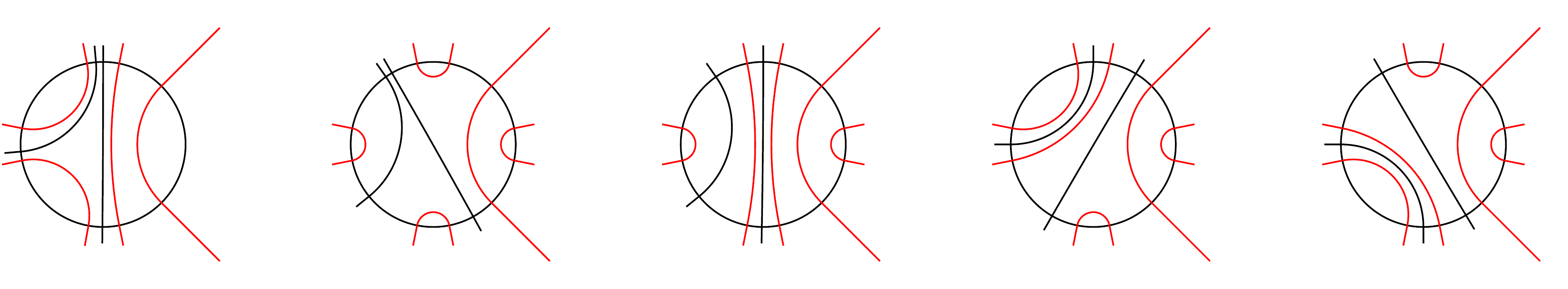}
\caption{One of the two monkey turnovers associated to the vertex $(1,1)$ 
is a $Y_2$ bundle over $S^1$, where $Y_2$ is built from five pieces 
associated to the configurations indicated in the figure. 
The first two pieces are $K(B_3,1)$s and the last three are
$K(\Z^2,1)$s. $C_1$ and its preimages with greater height 
than $C_2,C_3$ are in red.}\label{five_pieces}
\end{figure}

When $C_1$ is not antipodal, the leaves $C_2$ and $C_3$ may interact, 
and the topology is significantly more complicated. 
This component is also a bundle over $S^1$ whose fiber is a 
certain 4-manifold $Y_2$ that we call the {\em Star of David}
(the explanation for the name will come in \S~\ref{section:sausages}). 
It is built from five pieces; two of these pieces are homotopic to
trefoil complements (i.e.\/ they are $K(B_3,1)$s). 
The other three pieces are homotopic to tori, which attach to the other
components along a subspace homotopic to a wedge of two circles; in other
words this decomposition does {\em not} form an injective complex of $K(\pi,1)$s. 
In fact, the fundamental group of $Y_2$ is obtained from the free product 
of two $B_3$s by adding three commutation relations. 
The five pieces are illustrated in Figure~\ref{five_pieces}.
\end{example}

\subsection{Degree $d$}\label{subsection:degree_d}

Most of what we have done in this section generalizes to degree
$d$ readily. Set $\SS_d = X_d \times \R$, and order critical heights with
multiplicity so that $1=h_1 \ge h_2 \ge \cdots \ge h_{d-1}$. Define
$\rho:X_d \to \R^{d-2}$ with coordinates $t_j:=-\log_d h_j$ for
$j=2,\cdots,d-1$. The image of $X_d$ is the Weyl chamber $W$, which is partitioned
by walls $t_{ij} \in \Z$ where $t_{ij}:=t_i-t_j$ into the cells of a cell
decomposition $\tau$ with dual decomposition $\tau'$. If we identify $\R^{d-2}$
affinely with the subspace of $\R^{d-1}$ with coordinates summing to 0, 
then $\tau$ becomes the {\em simplectic honeycomb}; see e.g.\/ Coxeter \cite{Coxeter}. 
For example, in degree 5 the cells of $\tau$ are regular tetrahedra and octahedra, 
and the cells of $\tau'$ are regular rhombic dodecahedra.

Let $\kappa$ be the cellulation defined only by the subset of walls $t_{i1}$
and let $\kappa'$ be the dual cellulation. Then we have:

\begin{lemma}[Equivalent Cells]\label{lemma:degree_d_equivalent_cells}
Let $K$ and $T$ be cells of the cellulations $\kappa'$ and $\tau'$ associated to
a vertex $v$. Then the components of $\rho^{-1}(K)$ and of $\rho^{-1}(T)$ are
homeomorphic, and are isotopic inside $X_d$.
\end{lemma}

\begin{theorem}[Complex of spaces]\label{theorem:degree_d_complex}
$\SS_d$ is a complex of monkey prisms and monkey turnovers over a contractible
$\tilde{A}_{d-2}$-building.
\end{theorem}

\begin{theorem}[Homotopy dimension]\label{theorem:degree_d_homotopy_dimension}
$\SS_d$ has the homotopy type of a $(d-1)$-complex (i.e.\/ a complex of half the
real dimension of $\SS_d$ as a manifold).
\end{theorem}

The proofs are all perfectly analogous to the proofs of
Lemma~\ref{lemma:equivalent_cells}, Theorem~\ref{theorem:degree_4_complex_of_spaces}
and (the relevant part of) Theorem~\ref{theorem:S_4_K_pi_1}. 

\subsection{Tautological Elaminations}\label{subsection:degree_d_tautological}

It is straightforward to generalize Definition~\ref{definition:tautological}
to higher degree for the critical leaves of least height. Fix 
$C_1,C_2,\cdots, C_{d-2}$ at heights $h_1\ge h_2 \cdots h_{d-2}$, and let $C_{d-1}$ at height
$h_{d-2}-\epsilon$ vary. Every time $C_{d-1}$ collides with a leaf $P$ which is
a preimage of $C_j$ for $j<d-1$ we add $P^d$ to the tautological elamination.

It is harder to decide on a definition for the other critical leaves. This is
because the elamination associated to $C_j$ depends on the fixed locations
of $C_k$ with $k<j$ and an {\em equivalence class} of fixed locations of $C_k$
with $k>j$. We explain.

\begin{definition}[Degree $d$ Tautological Elaminations]
Fix a degree $d$ and an index $1 < i \le d-1$. Fix locations of leaves
$C_j$ for $j \ne i$ where the $C_j$ with $j<i$ have heights $h_1 \ge h_2 \ge \cdots h_{i+1}$,
and the $C_j$ with $j>i$ have height $0$. We shall define the leaves of the 
tautological elamination $\Lambda_T(C)$ associated to 
$C:=C_1,\cdots \hat{C}_i \cdots C_{d-1}$ of depth $n$. 
Insert $C_i$ somewhere at height $h_{i+1}-\epsilon$ 
compatibly with the other leaves, and construct the leaves of the dynamical
elamination associated to the critical data $C \cup C_i$ 
which are preimages of $C_j$ up to depth $n$. 
As we vary $C_i$, the leaves $C_j$ with $j<i$ stay fixed
but the $C_j$ with $j>i$ are pushed over $C_i$ and over preimages of higher
depth critical leaves. Whenever $C_i$ collides with a preimage $P$ of a
higher $C_j$ we add $P^d$ to the tautological elamination.
\end{definition}

The $C_j$ with $j>i$ are `hidden parameters'; we need them to determine the
location of the preimages of greater height, but they do not themselves
contribute any leaves to $\Lambda_T$.

As the angles of $C_j$, $j <i$ vary by a vector of parameters $t$
(and $C_j$, $j>i$ are pushed over by this motion) the tautological elaminations 
vary by the flow $\FF_t$.

Within each monkey prism the pinching is described by these tautological
elaminations. Let's fix a cell $\tau'$ dual to a vertex $v$ where $t_j = n_j$ 
and a monkey prism which is a component of $\rho^{-1}(\tau')$.
The way in which the fiber $\Omega_i$ over $C_{< i} \in E_{i-1}$
pinches depends on which component we are in; implicitly,
this choice of component determines an equivalence class of the location
of $C_j$ with $j>i$ and therefore determines a tautological elamination. The
depth $\le n_i-n_j$ preimages of the $C_j$ in the tautological elamination
describe the pinching of $\Omega_i$ as a function of $C_{<i}$.
The proof is perfectly parallel to that of Theorem~\ref{theorem:S_3_topology}.

\subsection{Completed Tautological Elamination}

Fix $C:=C_1,\cdots,\hat{C}_i,\cdots,C_{d-1}$ as above. It is possible to 
define a suitable `completion' of the tautological elamination $\Lambda_T(C)$ as
follows.

\begin{definition}[Completed Tautological Elamination]
Fix $d$ and $C$ as above. In the construction of the tautological elamination,
set the formal height of $C_i$ to be equal to $0$, and 
define $\L_n$ to be the set of leaves of the form $P^d$ where
$P$ is a depth $n$ preimage of $C_i$ that collides with $C_i$ itself.

Although they have height $0$, the $\L_n$ have a well-defined vein in $\D$. 
Note that some pairs of leaves of $\L_n$ cross each other in $\D$. 
Nevertheless we can think of $\L_n$ as a closed subset of the space of 
geodesic leaves in $\D$ and take the lim sup
$\L_\infty:=\limsup_{n \to \infty} \L_n$ (i.e.\/ there is a leaf in $\L_\infty$ for
each convergent sequence of leaves in a subsequence of the $\L_n$). Then we define the
{\em completed} tautological elamination associated to $C$ to be 
$\bar{\Lambda}_T(C):=\Lambda_T(C) \cup \L_\infty$.

The leaves of $\bar{\Lambda}_T(C) - \Lambda_T(C)$ 
are called {\em flat} since they have height 0, to distinguish them from the 
{\em ordinary} leaves of $\Lambda_T(C)$.
\end{definition} 

\begin{theorem}[Limit is lamination]
The vein of $\bar{\Lambda}_T(C)$ is a geodesic lamination (i.e.\/ leaves of
$\L_\infty$ do not cross $\Lambda_T(C)$ or each other).
\end{theorem}
The proof of this will appear in a forthcoming paper.

Pinching along $\bar{\Lambda}_T(C)$ is the same as pinching along
$\Lambda_T(C)$, since the flat leaves all have height zero, so do not actually
intrude into $\E$. However, it {\em does} make sense to pinch the closure 
$\bar{\E} \subset \C \cup \infty$ along $\bar{\Lambda}_T(C)$, 
exactly as before by cut and paste along the
tips of $\Lambda_T(C)$, and then by quotienting the endpoints of the flat leaves to
single points. Let's call the result $\bar{\Omega}_T(C)$.
Because we added limits in the definition of $\bar{\Lambda}_T(C)$,
$\bar{\Omega}_T(C)$ is Hausdorff.  It is a compactification of $\Omega_T(C)$
away from $\infty$, by locally connected spaces (isolated points or monotone 
quotients of circles).

Notice that this construction is non-vacuous even when $d=2$; it reproduces 
Thurston's quadratic geolamination \cite{Thurston_dynamics}, which is 
a proposed topological model for 
the boundary of the Mandelbrot set (proposed, since it is famously unknown 
if the Mandelbrot set is locally connected). 

Thus it seems reasonable to conjecture that the boundary components of 
$\overline{\Omega}_T(C)$ should parameterize (modulo the
question of local connectivity) the boundaries of the components of the complement of
$\SS_d$ in the slice associated to $C$. Compare with \cite{Blokh_et_al}.

\section{Sausages}\label{section:sausages}

In this section we introduce a completely new way to see the pieces in the
building decomposition of $X_d$ via algebraic geometry. It will turn out that the
monkey prisms and monkey turnovers in $X_d$ all become homeomorphic (after 
taking a product with an interval) to (rather explicit) complex affine 
varieties --- moduli spaces of certain objects called {\em sausage shifts}.

\subsection{Sausages: the basic idea}

Everyone likes sausages. Now we will see them made. The basic idea is illustrated
in Figure~\ref{making_sausages}.

\begin{figure}[htpb]
\centering
\includegraphics[scale=0.5]{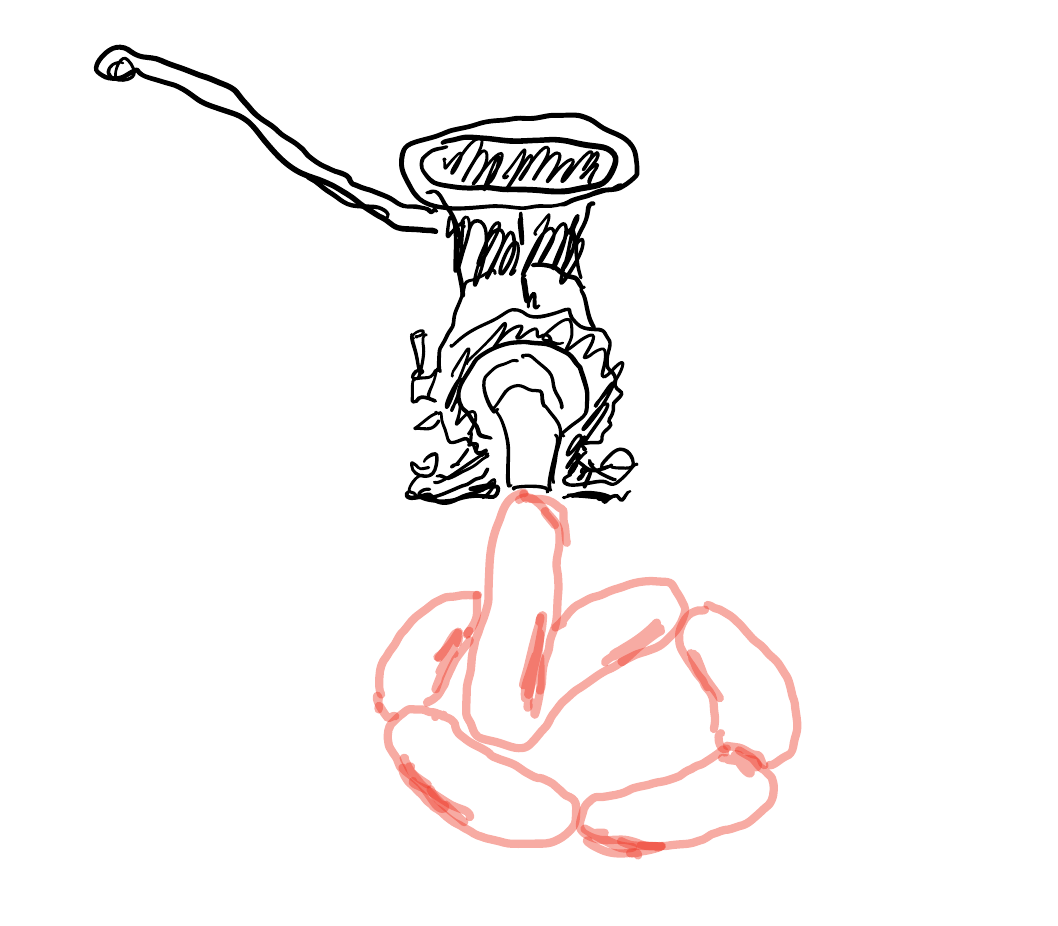}
\caption{making sausages}\label{making_sausages}
\end{figure}

A dynamical elamination is a machine that, by a process of repeatedly pinching
leaves in order of height, extrudes a long, complicated Riemann surface $\Omega$ 
(a Fatou set); by tying this Riemann surface off at periodic values of $-\log_d h$, 
we decompose it into manageable genus zero chunks: sausages.

Thus the Riemann surface $\Omega$ is tied off into a tree of sausages, and
the dynamics of $F$ on $\Omega$ decomposes into polynomial maps between the
sausages, whose moduli spaces are described by (elementary) algebraic geometry.

\subsection{Definitions}

\subsubsection{Tagged Points}

Let $f$ be a holomorphic map between open subsets of $\C$ taking $p$ to $q$. If
$f'(p)$ is nonzero, $df$ is a $\C$-linear isomorphism from $T_p$ to $T_q$. Thus
after scaling by a suitable positive real number, it induces an isometry of
unit tangent circles. We denote these unit tangent circles by $U$ and the induced
map as $Uf:U_p \to U_q$.

If $p$ is a critical point of multiplicity $m$, then $f$ maps infinitesimal
round circles centered at $p$ to infinitesimal round circles centered at $q$
by a degree $(m+1)$ covering. By abuse of notation we write $Uf:U_p \to U_q$ for this
map. In holomorphic coordinates for which $f$ is $z \to z^{m+1}$ this map is just
multiplication by $(m+1)$ on $U_0$ (really we are using an implicit identification
between the tangent space $T_q$ and its $(m+1)$st tensor power).

\begin{definition}[Tagged Point]
A {\em tagged point} is a point $p$ together with an element $u_p \in U_p$.
The {\em zero tag} is the point $0 \in \C$ together with the unit 
vector $u_0 \in U_0$ tangent to the positive real axis.

If $f$ is a holomorphic map taking a tagged point $p$ to a tagged point $q$ we say
it {\em preserves tags} if $Uf(u_p)=u_q$. If $f$ is a holomorphic function,
a {\em tagged root} is a tagged point $p$ with $f(p)=0$ for which $Uf(u_p)$ is
the zero tag.
\end{definition}

\subsubsection{Sausages}

Let $T$ be a locally finite rooted tree. Every vertex $v$ but the root has a unique
{\em parent} --- the unique vertex adjacent to $v$ on the unique embedded path in $T$
from $v$ to the root. If $w$ is the parent of $v$ we say $v$ is a child of $w$. Every
edge of $T$ is {\em oriented} from child to parent.

\begin{definition}[Bunch of sausages]
Let $T$ be a locally finite rooted tree. A {\em bunch of sausages} over $T$ is an
infinite nodal genus 0 Riemann surface $S$ made from a copy of $\CP^1$ for each vertex
$v$ of $T$ (the {\em sausages}, which we denote $\CP^1_v$) and for each $v$ a finite set of 
{\em marked tagged points} $Z_v \subset \CP^1_v - \infty$ and a bijection $\sigma$
from the children of $v$ to the set $Z_v$, so that if $w$ is a child of $v$, the
point $\infty$ in the sausage $\CP^1_w$ is attached to the point
$\sigma(w) \in Z_v \in \CP^1_v$.
\end{definition}

If $T$ is a rooted tree, for each vertex $w$ of $T$ there is a rooted subtree
$T_w \subset T$ with root $w$. If $S$ is a bunch of sausages over $T$, then 
$S_w \subset S$ denotes the bunch of sausages associated to the subtree $T_w$.

A {\em morphism} between rooted trees $T,T'$ is a simplicial map $\tau:T \to T'$ taking roots
to roots, and directed edges to directed edges. Thus if $w$ is a child of $v$,
the image $\tau(w)$ is a child of $\tau(v)$.

\begin{definition}[Augmentation]
If $T$ is a rooted tree, the {\em augmentation} of $T$, denoted $T'$, is 
the rooted tree obtained from $T$ by adding
a new root $v'$ and an edge from the root $v$ of $T$ to $v'$. If $S$ is a bunch of sausages
over $T$, the {\em augmentation} of $S$, denoted $S'$, is the bunch of sausages 
over $T'$ obtained by attaching $\CP^1_{v'}$ along $0 = Z_{v'}$ to $\infty$ in $Z_v$.
\end{definition}

\begin{definition}[Polynomial]\label{definition:polynomial}
Let $S$ be a bunch of sausages over a locally finite tree $T$. A {\em degree $d$ 
polynomial} $p$ is a degree $d$ tagged holomorphic map from $S$ to its augmentation $S'$ 
over a morphism $\tau:T \to T'$. This means that for every vertex $w$ of $T$
there is a polynomial map $p_w:\CP^1_w \to \CP^1_{\tau(w)}$ of degree $d_w$ in  
normal form taking $Z_w$ to $Z_{\tau(w)}$, and so that
\begin{enumerate}
\item{if $v$ is the root, the polynomial $p_v$ has degree $d$ and its roots are
exactly $Z_v \subset \CP^1_v$, and furthermore as tagged points 
$Z_v$ are tagged roots of $p_v$;}
\item{the root polynomial $p_v$ has more than one root; i.e.\/ $p_v$ is not
the polynomial $z^d$;}
\item{for every vertex $w$ with $\tau(w) = u$ the map $p_w:\CP^1_w \to \CP^1_u$ 
takes $Z_w$ to $Z_u$ as tagged points, and $Z_w$ is the entire preimage
$p_w^{-1}(Z_u)$; and}
\item{if $w$ is the child of $u$ with $\sigma(w) = z \in Z_u \subset \CP^1_u$ then
the degree $d_w$ of the polynomial $p_w$ is equal to the
multiplicity of $z$ as a preimage under $p_u$.}
\end{enumerate}
\end{definition}

The second bullet point is a kind of nondegeneracy condition: if the root
polynomial $p_v$ were $z^d$, then $S$ would already be the augmentation of some 
other sausage polynomial.

\begin{lemma}
Let $S$ be a bunch of sausages over $T$, and let $p:S\to S'$ be a degree $d$ polynomial
over a morphism $\tau:T \to T'$. Then for every vertex $w' \in S'$ the sum of degrees 
$\sum_{\tau(w)=w'} d_w = d$, and every point in $S'$ has exactly $d$ preimages, 
counted with multiplicity.
\end{lemma}
\begin{proof}
This is true for the root vertex by bullet (1) from Definition~\ref{definition:polynomial},
and by induction by bullets (2) and (3).
\end{proof}

This lemma justifies the terminology `polynomial map'.

\begin{definition}
Let $S$ be a bunch of sausages over $T$, and $p$ a polynomial map of degree $d$. 
Let $w$ be a vertex of $T$, and let $c \in \CP^1_w-\infty$ 
be a critical point for $p_w$. We say $c$ is a {\em genuine} critical point if
one of the following occurs:
\begin{enumerate}
\item{$c$ is not in $Z_w$; or}
\item{$c$ is in $Z_w$ but more than one sausage is attached at $c$;}
\end{enumerate}
and is {\em false} otherwise. In the second case, the multiplicity of $c$ is
equal to one less than the number of sausages attached at $c$.

We say $p$ is a {\em degree $d$ shift polynomial} and $(S,p)$ is a {\em degree $d$
sausage shift} if there are exactly $d-1$ genuine critical
points, counted with multiplicity.
\end{definition}

Bullet (2) in the Definition~\ref{definition:polynomial} is equivalent to saying 
that the root sausage contains at least one genuine critical point.

If $p$ is a shift polynomial, there is a minimal finite rooted subtree $U \subset T$
containing all the genuine critical points. Thus for $w \in T-U$, 
every polynomial $p_w$ is degree $1$; since it is in normal form it is the identity map
$p_w(z)=z$.

\begin{corollary}
Let $S$ be a bunch of sausages over $T$, and let $p$ be a degree $d$ shift polynomial. Then
the space $\EE(T)$ of ends of $T$ is a Cantor set, and the action of $p$ on $\EE(T)$ is
conjugate to the one-sided shift on right-infinite words in a $d$-letter alphabet.
\end{corollary}

\subsubsection{Isomorphism of polynomials}

The definition of a sausage polynomial includes data in the form of tags that
is essential if we want to construct a map from sausage polynomials to shift 
polynomials, as we shall do in \S~\ref{subsection:sausage_map}. 
In order for this map to be injective we must quotient out by a 
(finite) equivalence relation that we now explain.

Let $S$ be a bunch of sausages over a tree $T$, and let $p$ be a degree $d$
polynomial as in Definition~\ref{definition:polynomial}. Let $u$ be a vertex
of $T$, let $z \in Z_u \subset \CP^1_u$, and let $w$ be the child of $u$ with
$\sigma(w)=z$. If $z$ is a critical point of $p_u$ of multiplicity $m$ then
$p_w$ has degree $m+1$; i.e.\/ the degree of $p_u$ near $z$ agrees with the
degree of $p_w$ near infinity. In the sequel we will `cut open' $\CP^1_u$ at
$z$ and $\CP^1_w$ at infinity, and sew together the two resulting boundary 
circles in a dynamically compatible way, lining up the tag at $z$ in $\CP^1_u$
with the positive real axis at infinity in $\CP^1_w$. 

The tag at $z$ maps under $p_u$ to the tag at $p_u(z)$; thus given $p_u$
and the choice of tag at $p_u(z)$ we have freedom in the choice of a compatible
tag at $z$: different choices differ by multiplication by an 
$(m+1)$st root of unity $\zeta$. If we multiply the tag at $z$ by $\zeta$, 
we must at the same time change the coordinates on $\CP^1_w$ by multiplication
by $\zeta$. Changing coordinates on $\CP^1_w$ inductively affects the data
associated to $w$ and the subtree $T_w$ and its preimages under $p$ in the
obvious way. For example, $p_w(z)$ is replaced by $p_w(\zeta^{-1}z)$, the
marked points $Z_w$ are replaced by their preimages $\zeta Z_w$, etc. 

We say two sausage shifts are {\em isomorphic} if they are related by a
finite sequence of modifications of this sort. There are 
$\prod_{w\in T}\prod_{z\in Z_w} (m(z)+1)$
polynomials in an isomorphism class, where $m(z)$ is the multiplicity of $z$
as a critical point of $p_w$, and where the product is taken over all $z \in T_w$ 
for all $w \in T$. Note that for a sausage shift, this product is finite, since all
but finitely many $p_w$ have degree $1$.

\subsection{Moduli spaces}\label{subsection:moduli_spaces}

For each fixed combinatorial type of degree $d$ sausage shift, there is
an associated {\em moduli space} of isomorphism classes with the given
combinatorics, parameterized locally by the coefficients of the 
vertex polynomials $p_w$ of degrees $>1$. 
We shall see in Theorem~\ref{theorem:moduli_spaces} that 
moduli spaces for sausage shifts with generic heights have complex 
dimension $d-1$, and in fact they have the natural
structure of iterated bundles of complex affine varieties in an obvious way.

This is best explained by examples.

\begin{example}[Degree 2]
The root polynomial $p_v$ is of the form $z^2-c$ for some nonzero $c$. 
Since every other polynomial has degree $1$ (and
is therefore the identity function $z$) $S$ is a rooted dyadic tree, where each
parent has two children attached at $\pm \sqrt{c}$. The moduli space
of such sausages is evidently $\C^*$. This is homeomorphic (but {\em not} 
holomorphically isomorphic) to $\SS_2$.
\end{example}

\begin{example}[Distinct roots]
The simplest case in every degree $d$ is that the root polynomial $p_v$ has distinct
roots. Then every other polynomial has degree $1$ and $S$ is a rooted $d$-adic tree,
where each parent has $d$ children attached at the roots of $p_v$. Thus the
moduli space is a discriminant complement, and hence a $K(B_d,1)$.
\end{example}

\begin{example}[Degree 3]\label{example:degree_3_sausages}
Suppose the root polynomial $p_v$ has two roots, so it is of the form
$p_v:=(z-c)^2(z+2c)=z^3-3c^2z+2c^3$ with $c$ nonzero. The root vertex $v$ has two
children $u,w$ where $u$ is attached at the double root $c$ (say). Then
$p_w = z$ and $p_u$ has degree 2. Either $0$ is a genuine critical point for
$p_u$, or $p_u$ is of the form $z^2+c$ or $z^2-2c$. In the latter case
$u$ has two children $u',w'$ where $u'$ is attached at $0$ and this chain of 
critical roots $u,u',u^{(2)},u^{(3)},\cdots$ continues until $p_{u^{(n)}}:=z^2+x$ 
has a genuine critical point (or equivalently, $x \in \C - Z_t$ where $p$
takes the vertex $u^{(n)}$ to $t$). If we ignore tags, the moduli space 
is a bundle over $\C^*$ (parameterized by the choice of $c$) and whose fiber
is $\C - Z_t$. 

Notice that the points of $Z_t$ are obtained from $c, -2c$ by repeatedly
pulling back under double branch covers of the form $z \to z^2+c_j$ where
$c_j$ is one of the preimages pulled back so far. The monodromy acts on each
of these double branch covers either trivially or by permuting some of the
preimages in pairs. It follows that every orbit of the monodromy on $Z_t$
has length a power of $2$.
\end{example}

\begin{example}[Star of David]\label{example:moduli_star_of_david}
Suppose that the root polynomial in degree $4$ has one simple root and one triple
root; i.e.\/ the root polynomial is $p_v:=(z-c)^3(z+3c)$ with $c$ nonzero.
The root has two children $u,w$ where $u$ is attached at the triple root $c$
(say). The simplest case is when $c$ and $-3c$ are regular values for $p_u$. Then
the moduli space is a bundle over $\C^*$ whose fiber is a copy of $Y_2$, the space
of degree 3 polynomials $z^3+pz+q$ for which two specific distinct complex numbers 
(in this case $c$ and $-3c$) are regular values. It turns out that this moduli
space is homotopic to the monkey turnover described in Example~\ref{example:lamination_star_of_david}.
\end{example} 

The general structure of moduli spaces should now be starting to become clear.
To make a precise statement, we introduce the notion of a {\em Hurwitz Variety}:

\begin{definition}[Hurwitz Variety]
A {\em degree $d$ Hurwitz variety} is an affine complex variety of the following form.
Fix a finite set $Q \subset \C$ and a conjugacy class of representation
$\sigma$ from $\pi_1(\CP^1-Q)$ to the symmetric group $S_d$. 

The {\em Hurwitz Variety}
$H(Q,\sigma,d)$ is the space of degree $d$ normalized polynomials of the form 
$f(z):=z^d + a_2z^{d-2} + \cdots + a_d$ for which $f:\C \to \C$ is a degree
$d$ branched cover whose monodromy around $q$ is conjugate to $\sigma(q)$
for all $q \in Q$.
\end{definition}

For a permutation $\sigma$ let $|\sigma|=d-\text{number of orbits}$. 
Thus $|\sigma(q)|$ is the multiplicity of $q$
as a critical value of $f$, for each $q \in Q$ and each $f \in H(Q,\sigma,d)$.
We establish some basic properties of these varieties:

\begin{proposition}[Basic Properties]\label{prop:Hurwitz_basic}
Hurwitz varieties $H(Q,\sigma,d)$ satisfy the following basic properties:
\begin{enumerate}
\item{the dimension of $H(Q,\sigma,d)$ is equal to $d-1-\sum_q|\sigma(q)|$;}
\item{$H(Q,\sigma,d)$ is connected if its dimension is positive;}
\item{if there is a homeomorphism from $\CP^1$ to $\CP^1$ taking $Q$ to $Q'$ and
conjugating $\sigma$ to $\sigma'$ then
$H(Q,\sigma,d)$ is homeomorphic to $H(Q',\sigma',d)$.}
\end{enumerate}
\end{proposition}
\begin{proof}
The first bullet (i.e.\/ dimension count) is elementary. 

If we choose a finite subset $P \subset \C - Q$ and extend $\sigma$ to $P$
then we can build a degree $d$ branched cover of $\CP^1$ over $P \cup Q$ with 
monodromy $\sigma$ at $P \cup Q$. The genus of this 
branched cover depends only on $\sigma$. Thus the family of
covers which are connected and genus $0$ form a bundle over the space of 
pairs $Q\cup P,\sigma$ of a particular combinatorial type,
and it is an exercise in finite group theory to show that these 
fibers are connected when they have positive dimension. Each
$H(Q,\sigma,d)$ is a finite branched cover of the associated fiber (the
Riemann surface determines the polynomial up to finite ambiguity); this
proves the second bullet.

To prove the third bullet, let's modify our homeomorphism $\varphi:\CP^1 \to \CP^1$
by an isotopy so that it is equal to the identity in a neighborhood of $\infty$,
and is $K$-quasiconformal for some $K$. For each $f \in H(Q,\sigma,d)$ we
can pull back the Beltrami differential $\mu:=\bar{\partial}\varphi/\partial \varphi$
to $f^*\mu$ and let $\phi:\CP^1 \to \CP^1$ uniquely solve the Beltrami equation for
$f^*\mu$, normalized to be tangent to the identity at infinity to second order.
Then $\psi(f):=\varphi f \phi^{-1}$ is a normalized polynomial, and by
construction it is in $H(Q',\sigma',d)$. Letting $f$ range over $H(Q,\sigma,d)$
defines a homeomorphism $\psi:H(Q,\sigma,d) \to H(Q',\sigma',d)$ as desired.
\end{proof}

\begin{example}[Discriminant Variety]
If we set $Q =\lbrace 0 \rbrace$ and $\sigma$ the map to the identity element, 
then $H(\lbrace 0\rbrace,\id,d)$ is the space of degree $d$ polynomials in 
normal form with simple roots. In other words, 
$H(\lbrace 0\rbrace,0,d)$ is the complement of the discriminant variety, 
and is a $K(B_d,1)$.
\end{example}

\begin{theorem}[Moduli spaces]\label{theorem:moduli_spaces}
Every moduli space of a degree $d$ sausage shift of a fixed combinatorial type
is an algebraic variety over $\C$ which has the structure of an iterated
bundle whose base and fibers are all Hurwitz varieties. Furthermore, it has 
dimension $d-1$.
\end{theorem}
\begin{proof}
Consider a vertex $w$ with parent $u$ and image $v=\tau(w)$. There is a polynomial 
$p_w:\CP^1_w \to \CP^1_v$ whose degree is equal to the multiplicity of $u$
as a preimage under $p_u$. The points $Z_w$ are the preimages of $Z_v$ under
$p_w$, and the number and multiplicity of these points depends on the monodromy
of $p_u$ as a branched cover around $Z_v$. Thus for a fixed combinatorial type,
the polynomials $p_w$ vary in a Hurwitz Variety whose data is determined by the
polynomials in vertices above $w$. Changing a tag changes the coordinates on the
Hurwitz variety by a (finite) automorphism. Thus the moduli
space is an iterated bundle as claimed.
\end{proof}

\subsubsection{$K(\pi,1)$s}

Hurwitz varieties can apparently be quite complicated, topologically. But
at least in low degree we have the following theorem, which is by no means 
obvious, and which I personally find rather startling:

\begin{theorem}[$\CAT(0)$ 2-complex]\label{theorem:Hurwitz_3_CAT_0}
Every connected Hurwitz variety $H(Q,\sigma,3)$ is a $K(\pi,1)$ with the
homotopy type of a locally $\CAT(0)$ 2-complex.
\end{theorem}
\begin{proof}
If any point in $Q$ is a critical value the dimension is $1$ or $0$ 
and $H$ is either homotopic to a graph or to a finite set of points. So the
only interesting case is when $Q$ is a finite set and $\sigma$ is the 
constant map to the identity permutation. In other words, if $|Q|=n$, then
$H(Q,\id,3)$ is the (two complex dimensional) 
space $Y_n$ of degree $3$ polynomials $z^3 + pz+q$ for which the points 
in $Q$ are regular values. We show these have the homotopy type of locally
$\CAT(0)$ 2-complexes (and are therefore $K(\pi,1)$s).

First we describe the topology. By the third bullet of 
Proposition~\ref{prop:Hurwitz_basic} we can take $Q$
to be the set of $n$th roots of unity. Then $Y_n = \C^2 - V$, where $V$ is
the hyperplane in $\C^2$ with coordinates $p,q$ for which 
$\prod_j (-4p^3-27(q-\zeta^j)^2) = 0$. 
By a linear change of coordinates, we can replace this hyperplane by
$\prod_j (x^3-(y-\zeta^j)^2)= 0$.

$V$ intersects the plane $x=0$ in exactly the $n$th roots of unity. We foliate
the complement of this plane by (real 3-dimensional) open solid tori 
$S^1 \times \C$ thought of as a bundle over the circle $|x|=t$, and let
$V_\epsilon$ denote the intersection with $V$. If we cutoff $|y|$ at some big $T$,
then we get another solid torus $|y|=T,|x|\le t$ and the union is an $S^3$.
When $|x|=\epsilon$ is small and positive, $V_\epsilon$ splits into a union of 
$n$ trefoils $T^j_\epsilon$ (in this $S^3$), each obtained as a narrow cable of the circle
$y=\zeta^j$. The part of $Y_n$ in the domain $|x| \le \epsilon$ is homotopic
to a wedge of $n$ copies of a $K(B_3,1)$, one for each trefoil.

When $2|x|^{3/2}=|\zeta^j-\zeta^k|$ the trefoils $T^j$ and $T^k$ intersect at
three points, and when $|x|$ increases past this value, they become linked.
There are no other intersections. The link of a crossing (in $\C^2$) is a Hopf
link, and the result of pushing across each such crossing attaches a space to
$Y_n$, homotopic to a 2-torus, attached along a subspace homotopic to a wedge of
two circles. In other words, it attaches a 2-cell, whose boundary kills the
relator which is the commutator of two meridian circles linking the trefoils
at the point of intersection.

For each pair of trefoils $T^j, T^k$, we may choose Garside generators for 
$\pi_1(S^3 - T^j)$ corresponding to these meridian circles (the Garside
presentation for $B_3$ is of the form $\langle a,b,c\; |\; ab=bc=ca \rangle$).
Thus each pair of
trefoils contributes a subgroup of $\pi_1(Y_n)$ of the form
$$\langle a,b,c, x,y,z \; |\; ab=bc=ca, xy=yz=zx, [a,x]=[b,y]=[c,z]=1\rangle$$
However if we follow this chain of relations around a sequence of three
trefoils $T^j, T^k, T^l$ for which $j,k,l$ are positively oriented in $\Z$ mod $n$
(say), the intersection points of each pair of trefoils is successively displaced
by a rotation so that the holonomy of this chain of displacements rotates one
third of the way around. Thus for a triple of trefoils with Garside generators
$(a,b,c)$, $(n,m,o)$ and $(x,y,z)$, the commutation relations take the form
$$[a,n], [b,m], [c,o], [n,x], [m,y], [o,z], [x,b], [y,c], [z,a]$$

Here is another way of packaging the same information. Build a graph with vertices
at the $3n$th roots of unity, and with edges straight line segments between each
pair of roots whose ratio is a 3rd root of unity. Then $\pi_1(Y_n)$ is generated
by the edges of this graph, with relations that each triple of edges that form
a (n equilateral) triangle are Garside generators for a $B_3$, and each pair of
disjoint edges commutes. Furthermore, $Y_n$ is homotopic to the presentation
$2$-complex associated to this presentation. We shall show this $2$-complex
(or: a closely related and homotopic complex) can be given a $\CAT(0)$ structure.

Actually, there is a beautiful trick, that I learned from Jon McCammond, 
arising from his work with Tom Brady \cite{Brady_McCammond} on the 
construction of $\CAT(0)$ {\em orthoscheme complexes} for (certain) 
braid groups. First replace each Garside presentation 
$\langle a,b,c \; | \; ab=bc=ca\rangle$ by a presentation of the form
$\langle a,b,c,d \; | \; ab=bc=ca=d \rangle$. A presentation complex can be
built from three triangles with edges $abd^{-1}$ etc. The trick is to make these
{\em right angled} regular Euclidean triangles --- i.e.\/ to set the lengths of
$a,b,c$ to be $1$, and the length of $d$ to be $\sqrt{2}$. Let $K$ denote the
resulting complex (see Figure~\ref{CAT0_spine}), 
and let $K'$ be the complex built from $n$ copies of $K$
(one for each $B_3$) and one Euclidean square with edge length $1$ for each
commutation relation as above. 
We claim the resulting complex is $\CAT(0)$.

\begin{figure}[htpb]
\centering
\includegraphics[scale=0.8]{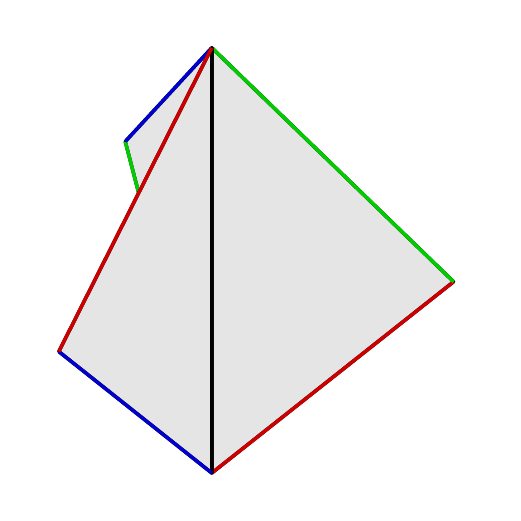}
\caption{$K$ is obtained from this complex by gluing free edges with the same
colors in pairs.}
\label{CAT0_spine}
\end{figure}

Let's see why. The complex $K$ (and $K'$ for that matter) has one vertex; since
these complexes are 2-dimensional and Euclidean, we just need to check that the
link of the vertex has no loop of length $<2\pi$. The link $L$ of the vertex of $K$
is a {\em theta graph}, with three edges of length $\pi$. The intersections with
the long edge $d$ are the vertices of the theta graph, and the intersections with
the edges $a,b,c$ give rise to six points (let's call these {\em short points}), 
each at distance $\pi/4$ from some vertex.

The link $L'$ of $K'$ is obtained from $n$ disjoint copies of $L$ by gluing a
4-cycle with edges of length $\pi/2$ for each commutation relation. Each such
4-cycle can be thought of as a complete bipartite graph on two sets of two points,
and each pair of points is attached to distinct short points in a copy of $L$. 
Since short points in $L$ are all distance $\pi$ apart, no cycle in the graph
associated to two $B_3$s and their commutators has length $<2\pi$. By the way, 
this shows that $\pi_1(Y_2)$ is $\CAT(0)$.

There is a simplicial map from $L'$ to the complete graph $K_n$
with edges all of length $\pi/2$ which just collapses each copy of $L$ to a
point, and identifies edges between the same pair of copies of $L$. 
A loop $\gamma$ in $L'$ of length $<2\pi$ would project to a (possibly immersed) 
simplicial `loop' in $K_n$ of simplicial length at most 3. If the projection 
has simplicial length 0 then $\gamma$ is contained in a copy of $L$ 
which we already know has no loops of length $<2\pi$. Simplicial length 1 
is impossible. If the projection of $\gamma$ has simplicial length $2$ in 
$K_n$ then $\gamma$ is contained in a subgraph formed
from a pair of copies of $L$ which (as we have just discussed) has no loops of
length $< 2\pi$. If the projection of $\gamma$ has simplicial length $3$ 
then it passes through a cycle of three $L$s, and because of the holonomy 
described above, a length $3\pi/2$ path in $\gamma$ has endpoints on the same
copy of $L$ but at {\em different} short points. Thus $\gamma$ has length
at least $2\pi$ and we are done.
\end{proof}

Together with Theorem~\ref{theorem:moduli_spaces} this immediately implies:

\begin{corollary}\label{corollary:degree_4_K_pi_1}
Every moduli space in degree $4$ is a $K(\pi,1)$.
\end{corollary}

\begin{question}
Is every Hurwitz Variety a $K(\pi,1)$? Is every Hurwitz Variety homotopic 
to a $\CAT(0)$ complex?
\end{question}

\subsection{The sausage map}\label{subsection:sausage_map}

Let $\hat{\SS}_d$ be the subspace of $\SS_d$ for which $\log_d(h_1) \in (-1/2,1/2)$,
where $h_1$ is the greatest critical height, and $\log_d$ denotes log to the base
$d$. This space is homeomorphic to $X_d \times (-1/2,1/2)$,
which is to say it is homeomorphic to $\SS_d$ itself. 

For $f \in \hat{\SS}_d$ let $L \in \DL_d$ be the dynamical elamination associated
to $f$ by the butcher map, and let $\Omega$ be the Riemann surface obtained by
pinching $L$ (so that $\Omega$ is canonically isomorphic to the Fatou set of $f$).

Let $\hat{\Omega}$ be the subspace of $\Omega$ with $\log_d(h)\le 1/2$ and
let $S$ be the quotient space of $\hat{\Omega}$ obtained by collapsing each component
with $\log_d(h)\in 1/2 + \Z$ to a point (which we call a {\em node}).

Each component $V$ of $S$ minus its nodes can be given a (branched) Euclidean
structure with horizontal coordinate $\theta$ and vertical coordinate
$\nu(h)$, where $\nu:\R^+ - d^{1/2 + \Z} \to \R$ is a function
that stretches each interval $(d^{n-1/2},d^{n+1/2})$ to $\R$ by a homeomorphism
(depending on $n$) in such a way that the map $z \to z^d$ on $\Omega$ 
is conformal in the new coordinates.

Let's explain this in terms of $\E$. In logarithmic coordinates $h,\theta$ we
can think of $\E$ as a half-open Euclidean cylinder which is the product of
the unit circle with the positive real numbers. The map $z \to z^d$
becomes multiplication by $d$, which we denote $\times d$. For each integer
$n$ let $I_n$ denote the open interval $(d^{n-1/2},d^{n+1/2})$ and let $A_n$
be the annulus in $\E$ where $h \in I_n$, and let $A:=\cup_n A_n \subset \E$.
Thus $\E-A$ is a countable set of circles with $\log_d(h)\in 1/2+\Z$. Thus
$\times d$ takes $A_n$ to $A_{n+1}$ for each $n$.

Choose (arbirarily) an orientation-preserving diffeomorphism $\nu_0:I_0 \to \R$
and for each $n$ define $\nu_n:I_n \to \R$ by $\nu_n(h):=d^n\nu_0(d^{-n}h)$.
Thus, by induction, $\nu_{n+1}(dh)=d\nu_n(h)$ for all $n$ and all $h \in I_n$.
Then define $\mu:A \to S^1 \times \R$ by $\mu(\theta,h)=(\theta,\nu_n(h))$
for $(\theta,h)\in A_n$. Thus $\mu$ semi-conjugates $\times d$ on $A$ to
$\times d$ on $S^1\times \R$. If we identify $S^1\times \R$ conformally with
$\C^*$ by exponentiating, then $\mu$ semi-conjugates $\times d$ on $A$ to
$z \to z^d$ on $\C^*$. If we keep a separate `copy' $\C^*_n:=\mu(A_n)$ for
each $n$, then we could say that $\mu$ conjugates $\times d$ on $A$ to
the self-map of $\cup_n \C^*_n$ that sends each $\C^*_n$ to $\C^*_{n+1}$ by
$z \to z^d$.


The components of $S$ minus its nodes are obtained from the $A_n$ by cut
and paste along segments of $L$, an operation which respects the 
Euclidean structure both in $h,\theta$ and $\nu(h),\theta$ coordinates

With respect to this branched Euclidean structure, the closure of each $V$ 
(i.e.\/ putting the nodes back in) is a compact Riemann
surface; in fact, it is isomorphic to $\CP^1$, and it is natural to choose
$\infty$ to be the (unique) node of greatest height. Thus $S$ becomes an
infinite nodal genus 0 Riemann surface. Furthermore although the quotient
map from $\hat{\Omega}$ to $S$ is very far from being holomorphic, the map
$z \to z^d$ on $\Omega$ does descends to a {\em holomorphic} map $p$ 
from $S$ to its augmentation giving $S$ the structure of a bunch of sausages,
and $p$ the structure of a degree $d$ shift polynomial. Notice that the
images of the critical points are precisely the genuine critical points of
the sausage polynomial.

Tags are defined at the nodes by identifying the unit tangent bundle at 
each node with a circle in $\Omega$, and inductively pulling back tags
compatibly with the dynamics of $z \to z^d$ so that the tag at the unique node in
the root of the augmentation corresponds to the argument $\theta=0$ (this
is well-defined, since $\theta$ takes values in $\R/\Z$ in the subspace of
$\Omega$ with $h$ greater than any critical height). 

\begin{theorem}[Sausage map]\label{theorem:sausage_map}
The sausage map is surjective, and is $1$--$1$ on the subspace of $\hat{\SS}_d$
for which no critical leaf $C_j$ has $\log_d(h_j) \in 1/2 + \Z$. This subspace 
maps bijectively to the set of isomorphism classes of degree $d$ sausage shifts.
\end{theorem}
\begin{proof}
It suffices to define a (continuous) inverse. Here is the construction. Cut open
a bunch of sausages along its set of nodes and sew in a copy of the unit tangent
circle $U$ at each point. Reparameterize the vertical coordinate on each component
by the inverse of $\mu$ (here we must choose the correct branch depending on the
combinatorial distance to the root). 
Each component becomes in this way a bordered Riemann surface.
The point $\infty$ in each $\CP^1_w$ gets a canonical
tag, namely the vector associated to the positive real axis. Thus we obtain a
collection of bordered surfaces, so that each border is a round circle with a tag,
and we glue these up respecting arguments and tags. By the definition of isomorphism,
the gluing is well-defined on an isomorphism class of sausage shift.
The result is a complete
planar Riemann surface $\Omega$ with one punctured end, and the sausage polynomial
descends to a degree $d$ self-map on $\Omega$ with $(d-1)$ critical points,
counted with multiplicity. By the Realization Theorem~\ref{theorem:realization}
this is the Fatou set of a unique shift polynomial.
 \end{proof}

\begin{theorem}[Monkey pieces are moduli spaces]\label{theorem:monkey_moduli}
The sausage map induces homeomorphisms from $(-1/2,1/2)$ times the open
monkey prisms and monkey turnovers arising in the decomposition in 
Theorem~\ref{theorem:degree_d_complex} to the moduli spaces of
degree $d$ sausage shifts of each fixed combinatorial type.
\end{theorem}
\begin{proof}
The factor of $(-1/2,1/2)$ comes from the difference between $\hat{\SS}_d$
and $X_d$, via orbits of the squeezing flow. 

This is a consequence of Theorem~\ref{theorem:sausage_map} and 
Lemma~\ref{lemma:degree_d_equivalent_cells}. 
Explicitly: the components of the images
of the sausage map are (up to this factor of $(-1/2,1/2)$) 
both the subspaces of $\rho^{-1}(W)$ in the preimage
of the cells $\kappa'$, and at the same time they are (by definition) the moduli
spaces of generic degree $d$ sausage shifts.
\end{proof}

Together with Corollary~\ref{corollary:degree_4_K_pi_1} and the discussion in
\S~\ref{subsection:K_pi_1} this completes the proof of Theorem~\ref{theorem:S_4_K_pi_1}. 
Moreover, together with Example~\ref{example:degree_3_sausages}, this completes 
the proof of Theorem~\ref{theorem:powers_of_2}.

\subsection{Sausages and combinatorics of the Tautological Elamination}\label{subsection:sausage_tautological}
We have already seen (Example~\ref{example:degree_3_sausages}) that moduli
spaces reveal nontrivial information about the tautological elamination.
Let $\Lambda_T$ denote the (depth 3) tautological elamination for some fixed 
$\theta_1$, and let $\Lambda_{T,n}$ denote the subset of leaves of depth $\le n$.
We have seen that monodromy permutes the components of $S^1$ mod $\Lambda_{T,n}$
in such a way that the orbits have length a power of 2.

We claim that these components all have {\em lengths} of the form $2^m/3^n$
for various $m$. Fix a sausage polynomial as in 
Example~\ref{example:degree_3_sausages} where the root vertex $v$ has
$Z_v=c,-2c$, and where there is a chain of vertices $u_1,\cdots,u_n$ mapping
to vertices $v=t_1,t_2,\cdots,t_n$ by polynomials $p_j:=z^2 + c_j$ so that
$0$ is a fake critical point for each $j<n$ (i.e.\/ $c_j \in Z_{t_j}$)
and a genuine one for $j=n$ ($c_n$ is not in $Z_{t_n}$).

The components of $S^1$ mod $\Lambda_{T,n}$ associated to sausages of this
combinatorial form are in bijection with the points of $Z_{t_n}$. Each
$w \in Z_{t_n}$ maps by a succession of polynomials of degrees 1 or 2 until
it reaches $c$ or $-2c$ (which themselves are mapped to $0$ by $p_v$).
The length of a component is multiplied by $1/3$ when we pull back a regular
value, and is multiplied by $2/3$ when we pull back a critical value. This
proves the claim.

Table~\ref{nltable} shows the number of 
components of length $\ell/3^n$ at each depth $n$ (omitted entries are zeroes).

\begin{table}[ht]
{\small
\centering
\begin{tabular}{c|c c c c c c c c c c c c c} 
 $n\backslash \ell$ & 1 & 2 & $2^2$ & $2^3$ & $2^4$ & $2^5$ & $2^6$ & $2^7$ & $2^8$ & $2^9$ & $2^{10}$ & $2^{11}$ & $2^{12}$\\
 \hline
 0 &  1\\ 
 1 &  1 & 1\\
 2 &  3 & 1 & 1\\
 3 &  7 & 6 & 0 & 1\\
 4 &  21 & 16 & 3 & 0 & 1\\
 5 &  57 & 51 & 13 & 0 & 0 & 1\\
 6 &  171 & 149 & 39 & 5 & 0 & 0 & 1\\
 7 &  499 & 454 & 117 & 23 & 0 & 0 & 0 & 1\\
 8 &  1497 & 1348 & 360 & 66 & 9 & 0 & 0 & 0 & 1\\
 9 &  4449 & 4083 & 1061 & 207 & 41 & 0 & 0 & 0 & 0 & 1\\
 10 & 13347 & 12191 & 3252 & 591 & 126 & 17 & 0 & 0 & 0 & 0 & 1\\
 11 & 39927 & 36658 & 9738 & 1799 & 370 & 81 & 0 & 0 & 0 & 0 & 0 & 1\\
 12 & 119781 & 109898 & 29292 & 5351 & 1125 & 240 & 33 & 0 & 0 & 0 & 0 & 0 & 1\\
 \hline
\end{tabular}
}
\vskip 12pt
\caption{Number of components of length $\ell/3^n$ at depth $n$}\label{nltable}
\end{table}

Note that there is a unique component with $\ell=2^n$ for each $n$;
this corresponds to the sausages for which $t_j=u_{j-1}$, $p_{u_1}=z^2+c$
and $p_{u_j}=z^2$ for $1<j<n$. The next biggest components have length 
$2^{\lfloor n/2 \rfloor}/3^n$.

The $(n,\ell)$ entry in this table is the number of components of length 
$\ell/3^n$ at depth $n$. If we denote this entry $N(n,\ell)$ then
$$\sum_\ell N(n,\ell) = (3^n+1)/2 \text{ and } \sum_\ell N(n,\ell)\cdot \ell = 3^n$$

\begin{example}[Recurrence]
Eric Rains observed the recurrence relation in the first column that
$$N(2n,1)=3\cdot N(2n-1,1) \text{ and } N(2n+1,1)=3\cdot N(2n,1)- 2\cdot N(n,1)$$
(a similar recurrence holds in higher degree).
The proof of this is surprisingly delicate, and will appear in a forthcoming
paper \cite{Calegari_combinatorics}.
\end{example}

\begin{example}[Short $\ell$ sequences]
One reason to be interested in the lengths of components of $S^1$ mod $\Lambda_{T,n}$
is that it gives us insight into the geometry of the {\em complement} of $\SS_3$.
Actually, it is easy enough to describe the picture in arbitrary degree.

For each degree $d$ the {\em shift complement} is $\C^{d-1}-\SS_d$. 
When critical points are simple, order them by height $h_1 \ge h_2 \ge \cdots h_{d-1}$, 
and define a {\em butcher's slice} $B(C_1,\cdots,C_{d-2})$ to be the subset of 
$\SS_d$ with $C_1,\cdots, C_{d-2}$ fixed and $h_{d-1}<h_{d-2}$. 
There is a tautological elamination $\Lambda_T(C_1,\cdots, C_{d-2})$
(see \S~\ref{subsection:degree_d_tautological}),
 and the result of pinching gives a Riemann surface 
$\Omega_T$ for which the subset of height $<h_{d-2}$ is holomorphically equivalent to 
$B$.

For the sake of simplicity, let's suppose $1=h_1=h_2=\cdots h_{d-2}$ so that the
leaves of $\Lambda_T$ of depth $n$ all have height $d^{-n}$.
A chain of successive components of $S^1$ mod $\Lambda_{T,n}$ with lengths
$\ell_n\cdot d^{-n}$ determines a system of disjoint annuli in the butcher's 
slice with moduli  $1/\ell_n$. So if $\sum_n 1/\ell_n$ diverges 
(for instance, if the sequence $\ell_n$ is bounded), 
the modulus goes to infinity and the end of $B$ converges to an 
{\em isolated} point in the complement of the shift locus. Call such an end of $B$
a {\em small end}. All but countably many of the (uncountable) ends of $B$ are small.

As we exit a small end of $B$, points in the Julia set collide in the limit 
to give rise to a non-shift Cantor Julia set 
(c.f.\/ Example~\ref{example:Cantor_non_shift}; also compare with
Branner \cite{Branner}). The local path 
component of the shift complement containing this limit 
point has complex dimension $d-2$, and is parameterized by the escaping 
critical points. There are uncountably many of these local path components, 
parameterized locally by the small ends of $B$.

Dragging critical points off to the (Cantor) Julia set one by one defines a nested
sequence of holomorphic submanifolds of the shift complement, each 
parameterized by the remaining escaping critical points.
When $C_{j+1}\cdots C_{d-2}$ have been dragged off to $J_f$, we can define a
butcher's slice by fixing $C_1,\cdots,C_{j-1}$ and letting $C_j$ vary; this
slice is the subset of height $<h_{j-1}$ in the Riemann surface $\Omega_T(C)$
associated to the tautological elamination $\Lambda_T(C)$ with critical data
$C:=C_1,\cdots,\hat{C}_j,\cdots C_{d-1}$ for a suitable equivalence class of
$C_{j+1},\cdots, C_{d-2}$ (see \S~\ref{subsection:degree_d_tautological}). 
Small ends of these butcher's slices locally parameterize the space of 
these $(j-1)$-dimensional submanifolds. 
\end{example}

\section{Fundamental Groups}\label{section:fundamental_groups}

\subsection{Braid Groups}

Let $\Delta_d \subset \C^{d-1}$ be the discriminant variety, parameterizing degree
$d$ polynomials in normal form $z^d + a_2 z^{d-2} + \cdots a_d$ with multiple roots.
The group $\pi_1(\C^{d-1}-\Delta_d)$ acts as permutations of these roots; the
permutation representation is a surjective map from $\pi_1(\C^{d-1}-\Delta_d)$ to the
symmetric group $S_d$. 

This map is very far from being injective. A loop in $\C^{d-1}-\Delta_d$ defines 
not just a permutation of roots, but a {\em braid}: the mapping class represented by
the combinatorial manner in which the points move around each other.
In other words, there is a {\em monodromy representation} 
$\Mon:\pi_1(\C^d - \Delta_d) \to B_d$ where $B_d$ is Artin's {\em braid group} on
$d$ strands. Forgetting the braiding determines a surjection $\Art:B_d \to S_d$.

Thus we obtain a factorization
$$\pi_1(\C^{d-1}-\Delta_d) \xrightarrow{\Mon} B_d \xrightarrow{\Art} S_d$$
where the first map is an isomorphism, and the second indicates that $B_d$ is
functorially obtained from $S_d$ by the algebraic process of {\em Artinization}.

\subsection{Shift automorphisms}

Let $\Sigma_d$ denote the space of right-infinite words on a $d$ letter alphabet; i.e.\/
$\Sigma_d:=\lbrace 1,\cdots,d\rbrace^\N$. This is a Cantor set in the product topology,
and the shift $\sigma$ acts as a $d$ to $1$ expanding map. Let $\hat{S}_d$ denote the
group $\hat{S}_d:=\Aut(\Sigma_d,\sigma)$; i.e.\/ the group of homeomorphisms of the Cantor
set commuting with the shift. 

In \cite{Blanchard_Devaney_Keen}, Blanchard--Devaney--Keen showed that the natural map
$\pi_1(\SS_d) \to \hat{S}_d$ is surjective, in every degree $d$. As before, this is
very far from being injective (as we shall shortly see).

Monodromy defines a representation
$\Mon:\pi_1(\SS_d) \to \Mod(\C - \text{Cantor set})$, but this map is certainly not
an isomorphism, since $\pi_1(\SS_d)$ is countable whereas $\Mod(\C-\text{Cantor set})$ has
the cardinality of the continuum. Actually, the image can be lifted to 
$\Mod(\text{Disk} - \text{Cantor set})$, since all shift polynomials (in normal form)
are tangent to second order near infinity. Let's denote the image by $\hat{B}_d$.

Forgetting the braiding defines a surjective
homomorphism $A:\Mod(\text{Disk}-\text{Cantor set}) \to \Aut(\text{Cantor set})$, and the
image of $\hat{B}_d$ is $\hat{S}_d$. I proved (see \cite{Calegari_planar}) that 
$\Mod(\text{Disk}-\text{Cantor set})$ is left-orderable, and therefore torsion-free,
whereas $\hat{S}_d$ is generated by torsion.

In any case we have a factorization of the Blanchard--Devaney--Keen map as
$$\pi_1(\SS_d) \xrightarrow{\Mon} \hat{B}_d \xrightarrow{A} \hat{S}_d$$
Neither map seems easy to understand. On the other hand, with Juliette Bavard and
Yan Mary He we were able to show:

\begin{theorem}[Bavard--Calegari--He]\label{theorem:monodromy_3}
In degree $3$ the map $\Mon:\pi_1(\SS_3) \to \hat{B}_3$ is an isomorphism.
\end{theorem}

The proof of this theorem shall (hopefully!) appear in a forthcoming paper.
The most optimistic conjecture I can make is:
\begin{conjecture}[Monodromy Conjecture]
The map $\Mod:\pi_1(\SS_d) \to \hat{B}_d$ is an isomorphism in every degree.
\end{conjecture}
The only real evidence I have in favor of this conjecture is that it is not
obviously falsified by the simplest cases I was able to fully analyze.

If $Y=H(Q,\sigma,e)$ is a Hurwitz variety, the preimage of $Q$ under 
$f \in Y$ is a finite subset of $\C$ whose cardinality is constant as a function of 
$f$, and therefore we obtain a monodromy map $M:\pi_1(Y) \to B_n$ for suitable $n$
depending on $Y$. If $Y$ is a Hurwitz variety that arises as a fiber of a
moduli space, the image of $\pi_1(Y) \to \pi_1(\SS_d) \to \hat{B}_3$ factors
through this $B_n$, so the monodromy conjecture implies that the maps $M$
are injective. In fact, at least in low dimensions, the monodromy conjecture
is {\em equivalent} to injectivity on these pieces, since both $\pi_1(\SS_d)$
and $\hat{B}_d$ are built up in understandable ways from these pieces 
(this is how Theorem~\ref{theorem:monodromy_3} is proved).

In any case, this is something we can test, since the groups $\pi_1(Y)$ and
$B_n$ are rather explicit, especially in low degree.

\begin{example}[Star of David]
The `hard' pieces in degree 4 are the Star of David and its generalizations 
as discussed in Theorem~\ref{theorem:Hurwitz_3_CAT_0}.

Recall the moduli space $Y_2$ from Example~\ref{example:moduli_star_of_david},
and the description of its fundamental group in 
Theorem~\ref{theorem:Hurwitz_3_CAT_0}. This fundamental group (let's call it $G$)
has a presentation
$$G:=\langle a,b,c,x,y,z \; | \; ab=bc=ca, xy=yz=zx, [a,x]=[b,y]=[c,z]=1 \rangle$$
The monodromy map to $B_6$ arises by thinking of the generators as the
edges of a Star of David in the plane, and taking each generator to the braid
that cycles the endpoints of the edge around each other in a narrow ellipse
contained in a neighborhood of the edge.

There is an isometric embedding from the $\CAT(0)$ complex for $G$ 
described in Theorem~\ref{theorem:Hurwitz_3_CAT_0} to the Brady--McCammond
complex for $B_6$, which has been shown to be $\CAT(0)$ by Haettel--Kielak--Schwer.
If the image is totally geodesic, this would imply that $G \to B_6$ is injective.
This seems quite plausible, but we have not checked it.
\end{example}

\section{Acknowledgments}

I would like to thank Laurent Bartholdi, Juliette Bavard, Pierre Deligne,
Laura DeMarco, Yan Mary He, Sarah Koch, Jeff Lagarias, Chris Leininger, Jon McCammond, 
Curt McMullen, Madhav Nori, Kevin Pilgrim, Eric Rains, Alden Walker, Henry Wilton
and the anonymous referee for their help. 
Most of what I know about polynomial dynamics (which is not
much) I learned from Sarah and from Curt at various points in time. 

I would also like to extend thanks to the students who attended the graduate topics course
I taught on this material at the University of Chicago in Winter 2019, and to
Sam Kim who solicited some talks and a paper for the celebration of the 25th 
anniversary of the founding of KIAS, and without whom I might have never been 
sufficiently motivated to write any of this up.

\end{document}